%% file: manuscript.tex
\documentclass{m2an}

\usepackage{amsmath}
\usepackage{amssymb,amsfonts,amsthm,bm}
\usepackage[usenames,dvipsnames]{color}
\usepackage{graphicx,caption,subfig}
\usepackage{xspace}
\usepackage{placeins}
\newtheorem{thm}{Theorem}[section]
\newtheorem{prop}[thm]{Proposition}
\newtheorem{lem}[thm]{Lemma}

\newtheorem{remark}[thm]{Remark}

\newcommand{\myinclude}[2]{{\includegraphics{pics/ellipses_a_N10_Ex3}}}

\newcommand{\sm}{\hspace{-2pt}-\hspace{-2pt}}
\renewcommand{\sp}{\hspace{-2pt}+\hspace{-2pt}}

\newcommand{\cK}{{\mathcal K}}

\newcommand{\jl}{[\![}
\newcommand{\jr}{]\!]}
\newcommand{\jmp}[1]{\jl#1\jr}
\newcommand{\al}{\{\hspace{-3.5pt}\{}
\newcommand{\ar}{\}\hspace{-3.5pt}\}}
\newcommand{\avg}[1]{\al#1\ar}
\newcommand{\abs}[1]{\left\vert#1\right\vert}
\newcommand{\wt}[1]{\widetilde{#1}}

\newcommand{\cd}{\hspace{-2pt}\cdot\hspace{-2pt}}
\newcommand{\sscal}[2]{(#1,#2)_{\star,\kappa}}

\newcommand{\tnorm}[1]{|\!|\!| #1|\!|\!|}
\newcommand{\snorm}[1]{\|#1\|_{\star,\kappa}}

\newcommand{\uN}{u_N}
\newcommand{\vN}{v_N}
\newcommand{\lN}{\lambda_N}

\newcommand{\real}{\mathbb R}
\newcommand{\Hper}{H^1_\#(\Omega)}
\newcommand{\VN}{\mathbb V_N}

\newcommand{\grad}{\nabla}

\newcommand{\pz}{\Pi^\kappa_0}
\newcommand{\pN}{\Pi_N^\kappa}

\newcommand{\bub}{g}

\newcommand{\etaR}{\eta_{{\tt R},\kappa}}
\newcommand{\etaF}{\eta_{{\tt F},\kappa}}
\newcommand{\etaJ}{\eta_{{\tt J},\kappa}}

\newcommand{\hotlbL}{{\tt hot}_\kappa^{\tt lb}}
\newcommand{\hotub}{{\tt hot}^{\tt ub}}
\newcommand{\hotlb}{{\tt hot}^{\tt lb}}
\newcommand{\ccR}{c_{{\tt R},\kappa}}
\newcommand{\ccF}{c_{{\tt F},\kappa}}
\newcommand{\ccJ}{c_{{\tt J},\kappa}}

\newcommand{\gk}{\gamma_\kappa}
\newcommand{\hgk}{\widehat{\gamma}_\kappa}

\newcommand{\dku}{{\tt d}_\kappa^u}
\newcommand{\dkpu}{{\tt d}_{\kappa'}^u}
\newcommand{\dkN}{{\tt d}_{\kappa}}
\newcommand{\rak}{{\tt a}_\kappa}
\newcommand{\rbk}{{\tt b}_\kappa}
\newcommand{\rbkp}{{\tt b}_{\kappa'}}
\newcommand{\rbok}{{\tt b}_{\omega(\kappa)}}
\newcommand{\rck}{{\tt c}_\kappa}

\makeatletter
\newcommand{\pushright}[1]{\ifmeasuring@#1\else\omit\hfill$\displaystyle#1$\fi\ignorespaces}
\newcommand{\pushleft}[1]{\ifmeasuring@#1\else\omit$\displaystyle#1$\hfill\fi\ignorespaces}
\makeatother

\begin{document}

\title{A posteriori error estimates for discontinuous Galerkin methods
using non-polynomial basis functions. \\Part II: Eigenvalue problems}
\author{Lin Lin}
\address{
Department of Mathematics, University of California, Berkeley
and Computational Research Division, Lawrence Berkeley National
Laboratory, Berkeley, CA 94720, USA. Email: linlin@math.berkeley.edu}
\author{Benjamin Stamm}
\address{
Center for Computational Engineering, Mathematics Department, RWTH Aachen University, Schinkelstr. 2, 52062 Aachen, Germany and Computational Biomedicine, Institute for Advanced Simulation IAS-5 and Institute of Neuroscience and Medicine INM-9, Forschungszentrum J\"ulich, Germany,
Email: stamm@mathcces.rwth-aachen.de}



\date{}

\begin{abstract}
We present the first systematic work for deriving a posteriori error estimates for general non-polynomial basis functions in an interior penalty discontinuous Galerkin (DG) formulation for solving eigenvalue problems associated with second order linear operators. Eigenvalue problems of such types play important roles in scientific and engineering applications, particularly in theoretical chemistry, solid state physics and material science. Based on the framework developed in [{\it L. Lin, B. Stamm, http://dx.doi.org/10.1051/m2an/2015069}] for second order PDEs, we develop residual type upper and lower bound error estimates for measuring the a posteriori error for eigenvalue problems.  The main merit of our method is that the method is parameter-free, in the sense that all but one solution-dependent constants appearing in the upper and lower bound estimates are explicitly computable by solving local and independent  eigenvalue problems, and the only non-computable constant can be reasonably approximated by a computable one without affecting the overall effectiveness of the estimates in practice. Compared to the PDE case, we find that a posteriori error estimators for eigenvalue problems must neglect certain terms, which involves explicitly the exact eigenvalues or eigenfunctions that are not accessible in numerical simulations. We define such terms carefully, and justify numerically that the neglected terms are indeed numerically high order terms compared to the computable estimators.  Numerical results for a variety of problems in 1D and 2D demonstrate that both the upper bound and lower bound are effective for measuring the error of eigenvalues and eigenfunctions.
\end{abstract}

\maketitle

\input{Introduction}
\input{Notation}

\input{EigenvalueProblem}
\input{NumericalResults}

\input{Conclusion}


\def\polhk#1{\setbox0=\hbox{#1}{\ooalign{\hidewidth
  \lower1.5ex\hbox{`}\hidewidth\crcr\unhbox0}}} \def\cprime{$'$}

\end{document}

%% file: Introduction.tex
\section{Introduction}
Let $\Omega$ be a rectangular, bounded domain. We consider the following linear
eigenvalue problem of finding an eigenvalue $\lambda$ and the corresponding eigenfunction $u$, with  $\| u\|_\Omega=1$, such that
\begin{equation}
  \label{eqn:problem}
	-\Delta u + V u = \lambda u,\qquad\mbox{in }\Omega, 
\end{equation}
where $V$
is a bounded, smooth potential.
Such eigenvalue problem arises in many scientific and engineering
problems. One notable example is the Kohn-Sham density
functional theory~\cite{KohnSham1965}, which is widely used in theoretical chemistry, solid state physics
and material science.
In order to solve Eq.~\eqref{eqn:problem} in practice, it is desirable
to reduce the number of degrees of freedom for discretizing
Eq.~\eqref{eqn:problem} to have a smaller algebraic problem to solve.
While standard polynomial basis functions and piecewise polynomial basis
functions can approach a complete basis
set and is versatile enough to represent almost any function of
interest, the resulting number of degrees of freedom is usually large even when
high order polynomials are used.  Non-polynomial basis functions are
therefore often employed to reduce the number of degrees of freedom.
Examples include the various non-polynomial basis sets used in
quantum chemistry such as the Gaussian basis
set~\cite{FrischPopleBinkley1984}, atomic orbital basis
set~\cite{Junquera:01}, adaptive local basis
set~\cite{LinLuYingE2012}, planewave discretization
\cite{CancesEtAl2012}. Similar techniques also appear in other contexts,
such as the planewave basis set for solving the eigenvalue
problems in photonic crystals~\cite{JoannopoulosJohnsonWinnEtAl2011},
Helmholtz
equations~\cite{HiptmairMoiolaPerugia2011,TezaurFarhat2006}, and 
heterogeneous multiscale method (HMM)~\cite{EEngquist2003} and the
multiscale finite element method~\cite{HouWu1997} for solving multiscale
elliptic equations.

\subsection{Previous work on a posteriori estimates}
Concerning the Laplace eigenvalue problem (Eq. \eqref{eqn:problem} with $V=0$), 
there has been important progress in particular
in obtaining guaranteed lower bounds for the first
eigenvalue using polynomial-based versions of the finite element method:
Armentano and Dur\'an~\cite{Arm_Dur_NC_eigs_LB_04},
Hu~{\em et al.}\ \cite{Hu_Huang_Lin_bounds_eigs_14,
Hu_Huang_Shen_eigs_NC_14}, Carstensen and
Gedicke~\cite{Cars_Ged_LB_eigs_14}, and Yang~{\em et al.}\
\cite{Yang_Han_Bi_Yu_CR_eig_adpt_15} achieve so via the lowest-order
nonconforming finite element method. Kuznetsov and
Repin~\cite{Kuzn_Rep_eigs_13}, and \v Sebestov\'a and
Vejchodsk\'y~\cite{Seb_Vej_eigs_bounds_14} give
numerical-method-independent estimates based on flux (functional)
estimates, and Liu and Oishi~\cite{Liu_Oish_bounds_eig_13} elaborate
a priori approximation estimates for lowest-order conforming
finite elements. 
Canc\`es~{\em et al.}\ \cite{CancelEtAl2015} present guaranteed bounds for the eigenvalue error for the classical conforming finite element method.
A posteriori estimates for the polynomial $hp$-discontinuous Galerkin (DG) method are developed by Giani and Hall~\cite{GianiHall2012}.
Earlier work comprises
Kato~\cite{Kat_bounds_eigs_49}, Forsythe~\cite{Fors_eigs_55},
Weinberger~\cite{Wein_eigs_56}, Bazley and
Fox~\cite{Baz_Fox_eigs_Sch_61}, Fox and
Rheinboldt~\cite{Fox_Rhei_eigs_66}, Moler and
Payne~\cite{Mol_Pay_bounds_eigs_68}, Kuttler and
Sigillito~\cite{Kut_Sig_eigs_78, Kut_Sig_eigs_book_85},
Still~\cite{Still_eigs_88}, Goerisch and He~\cite{Geor_He_eigs_89},
and Plum~\cite{Plum_bounds_eigs_97}.

The question of accuracy  for both eigenvalues and eigenvectors has
also been investigated previously. For conforming finite elements,
relying on the a priori error estimates resumed in Babu{\v{s}}ka and
Osborn~\cite{Bab_Osb_eigs_91}, Boffi~\cite{Boff_FE_eigs_00} and
references therein, a posteriori error estimates have been
obtained by Verf{\"u}rth~\cite{Verf_a_post_NL_ellipt_94}, Maday and
Patera~\cite{Mad_Pat_a_post_lin_out_00},
Larson~\cite{Lars_a_post_a_pr_FE_eig_00}, Heuveline and
Rannacher~\cite{Heu_Ran_a_post_FE_eig_02}, Dur{\'a}n~{\em et al.}\
\cite{Dur_Pad_Rodr_a_post_FE_eigs_03}, Grubi{\v{s}}i{\'c} and
Ovall~\cite{Grub_Ovall_est_eig_09}, Rannacher~{\em et al.}\
\cite{Ran_West_Woll_aig_disc_lin_10}, and Canc\`es~{\em et al.}\ \cite{CancelEtAl2015}.

For non-polynomial basis functions, the literature is much sparser. 
A posteriori estimates for planewave discretization of non-linear
Schr\"odinger eigenvalue problems are presented in Dusson and Maday
\cite{DussonMaday2013}, and Canc\`es~{\em et al.}\ \cite{CancEtAl2014}. 
Kaye~{\em et al.}\ \cite{KayeLinYang2015} developed upper bound error
estimates for solving linear eigenvalue problems using non-polynomial
basis functions in a DG framework, which generalizes the work of Giani~{\em et al.}\ \cite{GianiHall2012} for polynomial basis functions.

\subsection{Contribution}
We present a systematic way of deriving residual-based a posteriori
error estimates for the discontinuous Galerkin (DG) discretization of
problem \eqref{eqn:problem} using non-polynomial basis functions.  More
precisely, we derive computable upper and lower bounds for both the
error of eigenvalues and eigenvectors, up to some terms that are
asymptotically of higher order. This extends the framework introduced in
the companion paper~\cite{LinStamm2015} on second order PDEs.  
The main difficulty can be reduced to the
non-existence of inverse estimates for arbitrary non-polynomial basis
functions and the non-existence of an accurate conforming subspace.  
In the present approach, all but one basis-dependent constant appearing in the upper
and lower bound estimates are explicitly computable by solving local
eigenvalue problems.  For solutions with sufficient regularity (for
instance $u\in H^{2}(\Omega)$), the only non-computable constant can be
reasonably approximated by a computable one without affecting the
overall effectiveness of the estimates.  While the requirement of $H^{2}(\Omega)$
regularity appears to be a formal drawback in the context of a
posteriori error estimates, the main goal of this work is to develop a
posteriori error estimates for general basis sets rather than for
$h$-refinement, and the difficulty of general basis sets holds even if
the solution has $C^{\infty}(\Omega)$ regularity. Therefore we think our method can
have important practical values. 

Our estimators for eigenfunctions are very similar to those for 
second order PDEs, and our estimators for eigenvalues are derived
from the eigenfunction estimators. By leveraging the same constant related to
the regularity of the eigenfunction $u\in H^{2}(\Omega)$, we arrive at 
simpler treatment for upper and lower estimators for eigenvalues.
Compared to the treatment in literature~\cite{GianiHall2012}, our
treatment does not involve the usage of lifting operators.
Compared to the PDE case, we find that a posteriori error estimators for
eigenvalue problems must neglect certain terms, which involve
explicitly the exact eigenvalues and eigenfunctions that are not
accessible in numerical simulation. We define such terms carefully, and
justify numerically that the neglected terms are indeed high order terms
compared to the computable estimators.
Our numerical results in 1D and 2D
illustrate the effectiveness of the estimators.


%

\subsection{Outline} 
We introduce in Section 2 preliminary 
results that are needed to introduce the discontinuous Galerkin
discretization of the eigenvalue problem \eqref{eqn:problem} and the
following a posteriori analysis that are both presented in Section 3.
Section 4 is devoted to numerical tests, followed by the conclusion in Section 5.

%% file: Notation.tex
\section{Preliminary results}
\label{sec:Prelim}

\subsection{Mesh, broken spaces, jump and average operators}
Let $\Omega=(0,1)^d$, $d=1,2,3$ and let $\cK$ be a regular partition of
$\Omega$ into elements $\kappa\in\cK$.  That is, we assume that
the interior of $\overline{\kappa}\cap\overline{\kappa}'$, for any $\kappa,\kappa'\in \cK$, is either an
element of $\cK$, a common face, edge, vertex of the partition or the
empty set.
For simplicity, we identify the boundary of $\Omega$ in a periodical
manner. That means, that we also assume the partition to be regular
across the boundary $\partial \Omega$.  
We remark that although the assumption of a rectangular
domain with periodic boundary condition appears to be 
restrictive, such setup already directly finds its application in
important areas such as quantum chemistry and materials science.
However, the analysis below is not restricted to equations with periodic
boundary condition.  Other boundary conditions, such as Dirichlet or
Neumann boundary conditions can be employed as well with minor
modification.  Generalization to non-rectangular domain does not
introduce conceptual difficulties either, but may lead to changes in numerical
schemes for estimating relevant constants, if the tensorial structure of the grid
points is not preserved.

Let $N=(N_\kappa)_{\kappa\in\cK}$ denote the vector of the local number of degrees of freedom $N_\kappa$ on each element $\kappa\in\cK$.
Let $\VN=\bigoplus_{\kappa\in\cK}\VN(\kappa)$ by any piecewise discontinuous approximation space on a partition $\cK$ of the domain $\Omega$. 
It is important to highlight that 
little is assumed   about the a priori information of $\VN$
except that we assume that each $\VN(\kappa)$ contains
constant functions and that $\mathbb{V}_{N}(\kappa)\subset
H^{\frac32}(\kappa)$, so that the traces of $\grad v_N$ on the boundary
$\partial \kappa$ are well-defined for all $v_N\in\VN(\kappa)$, for all
$\kappa\in \cK$.

We denote by $H^s(\kappa)$ the standard Sobolev space of $L^2(\kappa)$-functions such that all partial derivatives of order $s\in\mathbb N$ or less lie as well in $L^2(\kappa)$.
By $H^s(\cK)$, we denote the set of piecewise $H^s$-functions defined by
\[
	H^s(\cK) = \left\{ v\in L^2(\Omega) \,\middle|\, v|_\kappa \in H^s(\kappa), \forall \kappa \in \cK \right\},
\]
also referred to as the broken Sobolev space.
We denote by $\Hper$ the space of periodic $H^1$-functions on $\Omega$.
We further define the element-wise resp. face-wise scalar-products and norms as
\[
	(v,w)_\cK = \sum_{\kappa\in\cK} (v,w)_\kappa
	\qquad\mbox{and}\qquad
	\|v\|_\cK = (v,v)_\cK^\frac12.
\]
The $L^2$-norm on $\kappa$ and $\Omega$ are denoted by $\|\cdot\|_\kappa$ and $\|\cdot\|_\Omega$, respectively.
The jump and average operators on a face $\overline{F}=\overline{\kappa}\cap\overline{\kappa}'$ are defined in a standard manner by
\begin{align*}
			&\avg{v} = \tfrac12(v|_{\kappa} + v|_{\kappa'}),
			\qquad&&\mbox{and}\qquad\qquad
			\jmp{v} = v|_{\kappa} n_\kappa + v|_{\kappa'} n_{\kappa'},\\
			&\avg{\grad v} = \tfrac12(\grad v|_{\kappa} + \grad v|_{\kappa'}),
			\qquad&&\mbox{and}\qquad\qquad
			\jmp{\grad v} = \grad v|_{\kappa} \cdot n_\kappa +\grad
      v|_{\kappa'} \cdot n_{\kappa'},
\end{align*}
where $n_\kappa$ denotes the exterior unit normal of the element $\kappa$. 
Finally we recall the standard result of piecewise integration by parts formula that will be employed several times in the upcoming analysis.
\begin{lem}
	\label{lem:IntByParts}
	Let $v,w\in H^{2}(\cK)$. Then, there holds
	\[
		\sum_{\kappa\in\cK} \Big[
			(\Delta v,w)_{\kappa} + (\grad v,\grad w)_{\kappa}
		\Big]
		= 
		\tfrac12
		\sum_{\kappa\in\cK} \Big[
			(\jmp{\grad v},w)_{\partial\kappa} + (\grad v,\jmp{w})_{\partial\kappa}
		\Big].
	\]
\end{lem}

\subsection{Projections}
For any element $\kappa\in\cK$, let us denote by $\pz:L^2(\kappa)\to \real$ the $L^2(\kappa)$-projection onto constant functions defined by
\[
	( \pz v, w)_\kappa = (v,w)_\kappa,\qquad \forall w\in\real,
\]
that is explicitly given by $\pz v =\frac{1}{|\kappa|} \int_\kappa v \,dx$.
On $H^1(\kappa)$ we define the following scalar product and norm
\begin{align}
	\label{eq:ScalProd}
	\sscal{v}{w}
	&= (\pz v,\pz w)_\kappa + (\grad v, \grad w)_\kappa,
	\\
	\nonumber
	\snorm{v}
	&= \sscal{v}{v}^\frac12,
\end{align}
for all $v,w\in H^1(\kappa)$
and the corresponding projection $\pN:H^1(\kappa)\to \VN(\kappa)$  by 
\begin{equation}
	\label{eq:DefProj}
	\sscal{\pN v}{w_N}
	=
	\sscal{v}{w_N}
	\qquad \forall w_N\in\VN(\kappa).
\end{equation}
Then, it is easy to see that this projection satisfies the following properties
\begin{align}
	\nonumber
	( v - \pN v, c)_\kappa &= 0, &&\forall c\in\real,\forall v\in H^1(\kappa),
\end{align}
or equivalently expressed as $\pz (v-\pN v)=0$.
This implies that
\begin{align}
	\label{eq:ProjOrtho}
	(\grad (v-\pN v),\grad w_N)_\kappa &= 0, && \forall w_N\in\VN(\kappa),\forall v\in H^1(\kappa),\\
	\label{eq:ProjStability}
	\|\grad (v-\pN v)\|_\kappa &\le \|\grad v\|_\kappa, && \forall v\in H^1(\kappa),\\
	\nonumber
	\snorm{v-\pN v} &\le \snorm{v}, && \forall v\in H^1(\kappa).
\end{align}

\subsection{Local scaling constants}
In this section, we recall some local constants that will be used in the upcoming a posteriori error analysis and that were introduced in \cite{LinStamm2015}.
We start with recalling the local trace inverse inequality constant $\dkN$ for each $\kappa\in\cK$ defined by
\[
	\dkN \equiv \sup_{v_N\in\VN(\kappa)} \frac{\|\grad v_N\cd n_\kappa\|_{\partial\kappa}}{\snorm{v_N}}>0.
\]
Further, we consider
\[
	\rak 
	\equiv \sup_{\substack{v\in H^1(\kappa),\\ v\perp\VN(\kappa)}} \frac{\|v\|_{\kappa}}{\snorm{v}}
	\qquad\mbox{and}\qquad
	\rbk \equiv \sup_{\substack{v\in H^1(\kappa),\\ v\perp\VN(\kappa)}} \frac{\|v\|_{\partial\kappa}}{\snorm{v}},
\]
where $\perp$ is in the sense of the scalar product $\sscal{\cdot}{\cdot}$ defined by \eqref{eq:ScalProd}.

\begin{remark}[The computation of the constants $\rak$, $\rbk$ and $\dkN$]
	More details on how these local constants can be approximated by solving local eigenvalue problems is explained in detail in \cite[Section 5]{LinStamm2015}
\end{remark}


%% file: EigenvalueProblem.tex
\section{Eigenvalue problem}
\label{sec:EigenVal}
We first assume that the smallest eigenvalue
$\lambda\in\real$ is non-degenerate.
Consider the problem of finding
this smallest eigenvalue $\lambda$ and the corresponding eigenfunction $u\in \Hper\cap H^2(\cK)$ with $\|u\|_\Omega=1$ such that 
\begin{equation}
	\label{eq:EigenVal}
	-\Delta u + V u = \lambda u,\quad\mbox{in } \Omega.
\end{equation}
We assume that $V$ is bounded and smooth.
Observing that adding any constant to the potential results in a
modified eigenvalue which is shifted by the same value, we can assume
that $V$ is positive. The choice of the constant only affects the
high order terms that is absent in the leading computable upper and
lower bound estimators.  

For some $\theta\in\mathbb R$ and $\gamma:\Omega\to \mathbb R$ such that $\gamma|_\kappa = \gk \in\real$ and using the bilinear form  
\begin{align*}
	a(w,v)
	&= 
		\sum_{\kappa\in\cK}
		\Big[
				(\grad w, \grad v)_\kappa 
				+ (Vw,v)_\kappa 
		\Big]
		+
		\tfrac12		
		\sum_{\kappa\in\cK}
		\Big[
				- (\grad w,\jmp{v})_{\partial\kappa} 
				- \theta (\jmp{w},\grad v)_{\partial\kappa} 
				+ \gk( \jmp{w},\jmp{v})_{\partial\kappa} 
		\Big],
\end{align*}
the approximated eigenvalue problem can be stated as: Find the smallest,
non-degenerate eigenvalue $\lN\in\real$ and $\uN\in \VN$ with $\|\uN\|_\Omega=1$ such that
\begin{equation}
	\label{eq:DG_EigenVal}
	a(\uN,\vN) = \lN (\uN,\vN)_\Omega,\qquad\forall\vN\in\VN.
\end{equation}

In order to quantify the error, we introduce the broken energy norm 
\[
	\tnorm{v}^2
	= \sum_{\kappa\in\cK}
		\Big[ \| \grad v \|_\kappa^2 +\tfrac{\gk}{2} \| \jmp{v} \|_{\partial\kappa}^2 + \|V^\frac12 v\|^2_\kappa\Big],\quad \forall v\in H^1(\cK).
\]
As usual, the penalty parameter $\gamma$ needs to be chosen sufficiently
large to ensure coercivity of the bilinear form, and the energy norm
error for eigenfunctions is defined to be $\tnorm{u-u_{N}}$.

Following the technique introduced in \cite{LinStamm2015} we obtain the following result. 
\begin{lem}
	If $\gk \ge \frac12 \,(1+\theta)^2 \,(\dkN)^2$, 
	then the bilinear form is coercive on $\VN$, i.e., there holds
	\[
		\tfrac12 \tnorm{v_N}^2 \le a(v_N,v_N), \qquad \forall v_N\in\VN.
	\]
\end{lem}
\begin{proof}
The proof is basically identical with the one presented in \cite[Lemma 3.1]{LinStamm2015}. 
The only slight difference is that the broken energy norm as well as the
bilinear form have now the positive contribution $\|V^\frac12
v_{N}\|^2_\kappa$.
\end{proof}

Remarkably, this lemma provides a computable and sharp value for each $\gamma_\kappa$ such that the bilinear form is coercive. 


\begin{remark}
  Note that even when the smallest eigenvalue $\lambda_{N}$ is a
  non-degenerate eigenvalue, the corresponding eigenfunction $u_{N}$ still
  has an arbitrary phase factor $\pm 1$. If such phase factors from $u$ and
  $u_{N}$ do not match, the error $u-u_{N}$ must be of order $1$. Since
  $u-u_{N}$ never appears in the upper or lower bound estimators, such
  phase factors will not affect the computation of the estimators, and only
  arise when comparing the estimators with the true error
  $\tnorm{u-u_{N}}$. In such case,
  the phase factor can be eliminated by 
  a ``subspace alignment'' procedure to be discussed
  in section~\ref{sec:numer}. The same procedure can be applied to align
  eigenfunctions when more eigenvalues and eigenfunctions are to be
  computed, even when some of the eigenvalues are degenerate. Below we
  assume that $u$ and $u_{N}$ are aligned eigenfunctions so that
  the error $\tnorm{u-u_{N}}$ converges to $0$ as the basis function
  refines.
\end{remark}

\subsection{A posteriori estimates of eigenfunctions}

We adapt here the residual type estimators obtained in
\cite{LinStamm2015} to the case of eigenvalue problems.

\subsubsection{Error representation}

Recall that we assumed that $u\in H^2(\kappa)$, we introduce the constant $\dku(u_N)$ defined by
\begin{align*}
	\dku(u_N) = \frac{\|\grad(u-u_N)\cd n_\kappa\|_{\partial\kappa}}{\|\grad(u-u_N)\|_{\kappa}},
\end{align*}
and define the constant $\rck$ by
\[
	\rck = \dku(u_N)+\dkN|\theta|.
\]
Without slight abuse of notation we may use $\dku\equiv \dku(u_{N})$, and neglect the
dependence on the numerical solution $u_{N}$.
We note that in practice, the constant $\dku(u_N)$ can not be evaluated
since $u$ is unknown. Theoretically the value $\dku(u_{N})$ can be
large. However, our previous numerical studies indicate that in many
cases $\dku(u_{N})$ can be relatively well approximately by the computable constant
$\dkN$.

We start with defining the residual type quantities 
\begin{align*}
	\etaR &\equiv \rak\|\lN\uN+\Delta \uN-V\uN\|_\kappa,\\
	\etaF &\equiv \tfrac{\rbk}2\|\jmp{\grad \uN}\|_{\partial\kappa},\\
	\etaJ &\equiv (\rbk\,\hgk+\tfrac{\rck}2)\|\jmp{\uN}\|_{\partial \kappa},
\end{align*}
where
\[
	\hgk = \max_{x\in\partial \kappa} \; \avg{\gamma}(x).
\]
Introducing the normalized error function
\[
	\varphi = \frac{u-\uN}{\tnorm{u-\uN}},
\]
and following the same strategy as in Section 3.2 of
\cite{LinStamm2015}, we develop the following {\bf error representation
equation}. 
\begin{align}
	\tnorm{u-\uN}
	&= 
	\label{eq:ErrRepEV}	
	\sum_{\kappa\in\cK}
	\Big[
	(\lN\uN+\Delta \uN- V\uN, \varphi\sm\varphi_N)_\kappa
	- \tfrac12(\jmp{\grad \uN},\varphi\sm\varphi_N)_{\partial\kappa} 
	\\
	\nonumber
	&\qquad\qquad
  - ( \avg{\gamma}  \jmp{\uN},(\varphi\sm\varphi_N)n_\kappa)_{\partial\kappa} 
	- \tfrac12(\jmp{\uN},\grad \varphi\sp\theta\grad \varphi_N)_{\partial\kappa} 
	\Big]
	+ \hotub,
\end{align}
for any $\varphi_N\in\VN$. In the following, we will use the particular choice $\varphi_N=\pN\varphi$.
%

The high order term for the upper bound estimator, denoted by
$\hotub$, is defined as
\[
	\hotub := (\lambda u-\lN\uN,\varphi)_\Omega.
\]
Using the normalization condition for eigenfunctions
$\|u\|_{\Omega}=\|\uN\|_{\Omega}=1$,  the term $\hotub$ can be simplified as 
\[
\hotub = \tfrac{\lambda+\lN}{2} (u-\uN,\varphi)_\Omega
	=\frac{\lambda+\lN}{2}\left(\frac{\|u-\uN\|_\Omega}{\tnorm{u-\uN}}\right)^2 \tnorm{u-\uN}.
\]
Asymptotically as $u_{N}$ converges to $u$,
$\frac{\|u-\uN\|_\Omega}{\tnorm{u-\uN}}$ characterizes the ratio between
the error measured in $L^{2}$ and $H^{1}$ norms, and converges to $0$.
Therefore $\hotub$ converges to $0$ faster than the energy norm
$\tnorm{u-\uN}$, and is neglected in the practically computed upper bound estimator.

\subsubsection{Upper bounds}
\begin{thm}
	\label{thm:upperBD}
	Let $u\in \Hper\cap H^2(\cK)$ be the solution of \eqref{eq:EigenVal} and $u_N\in\VN$ the DG-approximation defined by \eqref{eq:DG_EigenVal}. Then, we have the following a posteriori upper bound for the approximation error in the eigenfunction 
	\[
		\tnorm{u-\uN}
		\le 	  
		\left(
			\sum_{\kappa\in\cK}
			\Big[ \etaR+\etaF+\etaJ \Big]^2
		\right)^\frac12
		+ \hotub.
	\]
\end{thm}
\begin{proof}
One can proceeding as in the proof of Theorem 3.3 in \cite{LinStamm2015} based on the slightly modified error representation formula \eqref{eq:ErrRepEV}.

 \end{proof}


\subsubsection{Lower bounds}
We establish here lower bounds for the error in the eigenvector approximation following the strategy established in Section 4.2 of \cite{LinStamm2015}. 
We only explain the details when the technique differs in the case of eigenvalue approximations and summarize otherwise the results.
Observe that 
\[
	\etaJ
	= \left(\rbk\, \hgk +\tfrac{\rck}{2}\right)
	\|\jmp{\uN}\|_{\partial\kappa}
	\le  
	\sqrt{\tfrac{2}{\gk}}	\left(\rbk\, \hgk +\tfrac{\rck}{2}\right) \,\tnorm{u-u_N}_\kappa,
\]
and that
\[
	\etaF^2
    \le 
    	\tfrac{\rbk^2}{2} 
    	\Big(
		\max_{\kappa'\in\omega(\kappa)}\dkpu(u_N)
		\Big)^2
			\sum_{\kappa'\in\omega(\kappa)}  
				\|\grad (u-\uN)\|_{\kappa'}^2,
\]
where $\omega(\kappa)$ is the patch consisting of $\kappa$ and its adjacent elements sharing one face.

Further, let $\bub_\kappa$ be a smooth non-negative bubble
function with $\sup_{x\in\kappa}\bub_\kappa(x)=1$ and local support, i.e. $\mbox{supp}(\bub_\kappa)\subset \kappa$, which in turn implies that $\bub_\kappa|_{\partial\kappa} = 0$.
Let us denote the residual by $R=\lN\uN +\Delta u_N-Vu_N$ and define
\[
	\sigma_\kappa = \rak\frac{\|R\|_\kappa}{\|\bub_\kappa^\frac12R\|_\kappa^2}.
\]
Denote by $\varphi_\kappa\in H^1_0(\kappa)$ the local solution to 
Eq.~\eqref{eq:localbubble}
\begin{equation}
	-\Delta\varphi_\kappa = V\bub_\kappa R,\qquad\mbox{on } \kappa,
  \label{eq:localbubble}
\end{equation}
so that 
\begin{align*}
	\etaR 
	&= \rak\|R\|_\kappa 
	= \sigma_\kappa\|\bub_\kappa^\frac12R\|_\kappa^2
	=  \sigma_\kappa
		\int_\kappa \bub_\kappa \Big[-\Delta(u-u_N)+ V(u-\uN) +\lN\uN -\lambda u\Big]\, R
	\\
		&=-  \sigma_\kappa
		\int_\kappa \Big[\Delta(u-u_N)\,\bub_\kappa \,R - \Delta\varphi_\kappa(u-u_N)  + \bub_\kappa(\lN\uN -\lambda u) \, R\,\Big]
	\\
	&= \sigma_\kappa
			\int_\kappa 
			\Big[
			\nabla(u-u_N) \cdot \nabla (\bub_\kappa \, R)
			-
			\nabla(u-\uN)\cdot \nabla \varphi_\kappa 
			+ \bub_\kappa(\lN\uN -\lambda u) \, R
			\Big]
	\\ 
	&\le 
	 \sigma_\kappa
			\|\nabla(u-u_N)\|_\kappa
			 \|\nabla (\bub_\kappa \, R - \varphi_\kappa)\|_\kappa
			 +
			 \sigma_\kappa
			 \int_\kappa 
			\bub_\kappa(\lN\uN -\lambda u) \, R,
\end{align*}
and in consequence
\[
	\etaR
	\le 
	 \sigma_\kappa
			 \|\nabla (\bub_\kappa \, R - \varphi_\kappa)\|_\kappa \tnorm{u-u_N}_\kappa
       + \sigma_\kappa \|\lN\uN -\lambda u\|_\kappa \, \|
       \bub_{\kappa} R \|_\kappa.
\]

We define
\[
	\hotlbL = \|\lN\uN -\lambda u\|_\kappa \frac{\| \bub_{\kappa}R \|_\kappa}{ \|\nabla
  (\bub_\kappa \, R - \varphi_\kappa)\|_\kappa}.
\]
Numerical results indicate that $\hotlb$ can be much smaller compared to the
lower bound estimator as the basis set refines.
The results above indicate that
\begin{equation}
	\label{eq:locestim}
  \tnorm{u-u_{N}}_{\kappa} 
  \ge 
  \frac{\etaJ}{\ccJ},
  \qquad
 \tnorm{u-u_{N}}_{\omega(\kappa)} \ge 
  \frac{\etaF}{\ccF}
  \qquad\mbox{and}\qquad
    \tnorm{u-u_{N}}_{\kappa} +\hotlbL
  \ge 
  	\frac{\etaR}{\ccR}.
\end{equation}
where $|\omega(\kappa)|$ the cardinality of the set $\omega(\kappa)$,
and
\begin{align*}
    \ccR &= \rak \frac{\|R\|_{\kappa} \|\grad (\bub_\kappa
    \, R - \varphi_{\kappa})\|_\kappa}{\| \bub_\kappa^{1/2} \, R \|^2_{\kappa}},
    \\
   	\ccF &= \rbk \sqrt{\tfrac{|\omega(\kappa)|}{2}}
   	\max_{\kappa'\in\omega(\kappa)}\dkpu(u_N),
   	\\
    \ccJ &= \sqrt{\tfrac{2}{\gk}}	\left(\rbk\,\hgk+\tfrac{\rck}{2}\right).
\end{align*}

We summarize the results in the following proposition.
\begin{prop}[Local lower bound]
	Let $u\in \Hper\cap H^2(\cK)$ be the solution of \eqref{eq:EigenVal} and $u_N\in\VN$ the DG-approximation defined by \eqref{eq:DG_EigenVal}. 
	Then, the quantity
\[
	\xi_{\kappa}
	=	
	\max\left\{ 
	\frac{\etaR}{\ccR},
	\frac{\etaF}{\ccF},
	\frac{\etaJ}{\ccJ}
  \right\},
\]
is a local lower bound of the local error 
	\[
		\max\left\{\tnorm{u-u_{N}}_{\kappa} + \hotlbL,\tnorm{u-u_{N}}_{\omega(\kappa)}\right\}.
	\]
Here 
\[
	 \tnorm{v}_{\omega(\kappa)}^2 
	 = \frac{1}{|\omega(\kappa)|}\sum_{\kappa'\in\omega(\kappa)}  
				\|\grad v\|_{\kappa'}^2
		+ \tfrac{\gk}{2} \| \jmp{v}\|^2_{\partial\kappa},
\]
\end{prop}

\begin{remark}

In practice, sometimes both $\etaF$ and $\ccF$ can become very
small. Since $\ccF$ is computed inaccurately with iterative methods, the
ratio $\frac{\etaF}{\ccF}$ can become numerically unreliable. This can
be addressed by defining
\begin{equation}
	\xi_{\kappa} = \frac{\etaR+\etaF+\etaJ}{\ccR+\ccF+\ccJ}.
  \label{eqn:xikmod}
\end{equation}
Since
\[
\frac{\etaR+\etaF+\etaJ}{\ccR+\ccF+\ccJ} \le
	\max\left\{ 
	\frac{\etaR}{\ccR},
	\frac{\etaF}{\ccF},
	\frac{\etaJ}{\ccJ}
  \right\},
\]
Eq.~\eqref{eqn:xikmod} is still a local lower error bound, but is more
robust when $\ccF$ becomes small. Furthermore, among the three terms
$\frac{\etaR}{\ccR}, \frac{\etaF}{\ccF}, \frac{\etaJ}{\ccJ}$, one term
(usually the residual or the jump term)
is often in practice larger than the rest of the two terms combined. In this
case the use of \eqref{eqn:xikmod} leads to little loss of
efficiency. 
\end{remark}

On a global level, the following result holds.
\begin{prop}[Global lower bound]
	\label{prop:GlobLowerBound}
	Let $u\in \Hper\cap H^2(\cK)$ be the solution of \eqref{eq:EigenVal} and $u_N\in\VN$ the DG-approximation defined by \eqref{eq:DG_EigenVal}. 
	Then, there holds that
	\[
		\xi^2
		= \frac{
			\sum_{\kappa\in\cK} 
			\Big[ \etaR + \etaF + \etaJ \Big]^2
		}{
		3
		\max_{\kappa\in\cK}
			\left(
			\ccR^2 
			+
			{\rbok^2} \dku(\uN)^2
			+ 
			\ccJ^2
		\right)
		}
		\le 
		\tnorm{u- \uN}^2
		+
		(\hotlb)^2,
	\]
	where 
	\begin{align*}
		\rbok^2 
		&= \max_{F\in\partial\kappa}\avg{\rbk^2}|_F
		= \max_{F\in\partial\kappa}
			\left(\tfrac{\rbk^2}{2} + \tfrac{\rbkp^2}{2} \middle)\right|_F,
			\\
		\hotlb &= \left(\sum_{\kappa\in\cK}(\hotlbL)^2 \right)^\frac12.
	\end{align*}
\end{prop}
\begin{proof}
Observe that as explained in Section 4.2 of \cite{LinStamm2015}
\begin{align*}
	\sum_{\kappa\in\cK} \etaF^2
	&\le 
	\sum_{\kappa\in\cK} {\rbok^2} \dku(\uN)^2
		\|\grad(u- \uN)\|_{\kappa}^2,
\end{align*}
and then using the other local estimates for $\etaR$ and $\etaJ$ given by \eqref{eq:locestim} yields
\begin{align*}
	\sum_{\kappa\in\cK} 
	\Big[ \etaR + \etaF + \etaJ \Big]^2
	&\le 
	3\sum_{\kappa\in\cK} 
	\left(
	\etaR^2 + \etaF^2 + \etaJ^2
	\right)
	\\
	&
	\le 3
	\sum_{\kappa\in\cK} 
	\left(
		\ccR^2 
		+
		{\rbok^2} \dku(\uN)^2
		+ 
		\ccJ^2
	\right)
	\Big(\tnorm{u- \uN}_\kappa^2 + (\hotlbL)^2\Big)\\
	&
	\le 3
	\max_{\kappa\in\cK}
		\left(
		\ccR^2 
		+
		{\rbok^2} \dku(\uN)^2
		+ 
		\ccJ^2
	\right)
	\Big(\tnorm{u- \uN}^2 + (\hotlb)^2\Big).
\end{align*}
\end{proof}

\subsection{A posteriori estimates of eigenvalues}

Unlike the error of eigenfunctions $u-u_{N}$ of which the definition
requires a subspace alignment procedure, the definition of the error of
eigenvalues $\lambda-\lambda_{N}$ is directly well defined.  Our
strategy for obtaining the upper and lower bound estimators for
eigenvalues is to relate $\lambda-\lambda_{N}$ with the bilinear form
$a(u-u_{N},u-u_{N})$, and then bound errors of eigenvalues by errors of
eigenfunctions. Compared to treatment in
literature~\cite{GianiHall2012}, our treatment is slightly simpler and
does not involve lifting operators  due to regularity assumptions.

\begin{thm}
		Let $u\in \Hper\cap H^2(\cK)$ and $\lambda$ be the solution of \eqref{eq:EigenVal} and $u_N\in\VN$ and $\lN$ the DG-approximation defined by \eqref{eq:DG_EigenVal}. 
		Then, we have the following a posteriori upper bound for the approximation error in the eigenvalue

	\[
		|\lN - \lambda|
		\le			\max_{\kappa \in \cK}
		\left(
			1+\tfrac{\dku|1+\theta|}{2\gk^\frac12}
		\right) 
		\left(
		\eta
		+ \hotub
		\right)^2		+ \lambda\, \|u-\uN\|_\Omega^2,
	\]	
	where 
	\[
		\eta = 		
		\left(
			\sum_{\kappa\in\cK}
			\Big[ \etaR+\etaF+\etaJ \Big]^2
		\right)^\frac12.
	\]
\end{thm}
\begin{proof}
	Observe that
	\[
		a(u-\uN,u-\uN) 
		=
		\lambda + \lN -2a(u,\uN).
	\]
	We also use the fact that
	\[
		a(u,\uN) = \lambda (u,\uN),
	\]
	and that
	\[
		2\,(u,\uN) 
		= \|u\|^2_\Omega + \|\uN\|^2_\Omega - \|u-\uN\|_\Omega^2
		= 2 - \|u-\uN\|_\Omega^2,
	\]
	to derive
  \begin{equation}
		a(u-\uN,u-\uN) 
		=
    \lN - \lambda + \lambda\, \|u-\uN\|_\Omega^2.
    \label{eq:eigenval_eq}
  \end{equation}
	In consequence, we obtain the estimate
	\[
		|\lN - \lambda|
		\le |a(u-\uN,u-\uN)| + \lambda\, \|u-\uN\|_\Omega^2.
	\]
	Use that
	\[
		a(u-\uN,u-\uN)
		= 
		\sum_{\kappa\in\cK}\Big[
			\|\grad(u-\uN)\|^2_\kappa
			+ \|V^\frac12(u-\uN)\|^2_\kappa
		\Big]
		+\tfrac12
		\sum_{\kappa\in\cK}\Big[
			(1+\theta) (\grad(u-\uN),\jmp{\uN})_{\partial\kappa}
			+ \gk \|\jmp{\uN}\|^2_{\partial\kappa}
		\Big].
	\]
	The  Cauchy-Schwarz inequality and the definition of $\dku$ yields
	\begin{align*}
		(\grad(u-\uN),\jmp{\uN})_{\partial\kappa}
		&\le 
		\|\grad(u-\uN)\|_{\partial\kappa}
		\|\jmp{\uN}\|_{\partial\kappa}
		\le 
		\dku
		\|\grad(u-\uN)\|_{\kappa}
		\|\jmp{\uN}\|_{\partial\kappa},
	\end{align*}
	and thus
\begin{equation}
		a(u-\uN,u-\uN)
		\le 
		\sum_{\kappa\in\cK}\Big[
			\|\grad(u-\uN)\|^2_\kappa
			+ \|V^\frac12(u-\uN)\|^2_\kappa
		+\tfrac{\dku|1+\theta|}2
		\|\grad(u-\uN)\|_{\kappa}
		\|\jmp{\uN}\|_{\partial\kappa}
		+ \tfrac{\gk}{2} \|\jmp{\uN}\|^2_{\partial\kappa}
		\Big].
  \label{eqn:abound}
\end{equation}
	Applying now Young's inequality, we get
	\[
		\|\grad(u-\uN)\|_{\kappa}
		\|\jmp{\uN}\|_{\partial\kappa}
		\le \frac{1}{(2\gk)^\frac12} \|\grad(u-\uN)\|_{\kappa}^2
		+ \frac{(2\gk)^\frac12}{4} \|\jmp{\uN}\|_{\partial\kappa}^2.
	\]
  Inserting this into ~\eqref{eqn:abound} yields
 	\begin{align*}
		\lefteqn{a(u-\uN,u-\uN)}\\
		&\le 
		\sum_{\kappa\in\cK}\Big[
			\|\grad(u-\uN)\|^2_\kappa
			+ \|V^\frac12(u-\uN)\|^2_\kappa
		+\tfrac{\dku|1+\theta|}{(8\gk)^\frac12}
		 \|\grad(u-\uN)\|_{\kappa}^2
		+ \tfrac{\dku|1+\theta|(2\gk)^\frac12}{8} \|\jmp{\uN}\|_{\partial\kappa}^2		
		+ \tfrac{\gk}{2} \|\jmp{\uN}\|^2_{\partial\kappa}
		\Big]
		\\
		&= 
 		\sum_{\kappa\in\cK}\Big[
		\Big( 1+ \tfrac{\dku|1+\theta|}{(8\gk)^\frac12} \Big)
		\|\grad(u-\uN)\|^2_\kappa
			+ \|V^\frac12(u-\uN)\|^2_\kappa
		+ \tfrac{\gk}{2} \Big( 1+ \tfrac{\dku|1+\theta|}{(8\gk)^\frac12} \Big) \|\jmp{\uN}\|^2_{\partial\kappa}
		\Big]
			\\
		&\le 
		\max_{\kappa\in\cK} \Big( 1+ \tfrac{\dku|1+\theta|}{(8\gk)^\frac12} \Big) \tnorm{u-\uN}^2.
 	\end{align*}
	Applying now the result of Theorem \ref{thm:upperBD}, we get
	\begin{align*}
		|a(u-\uN,u-\uN)|
		&\le 
		\max_{\kappa \in \cK}
		\left(
			1+\tfrac{\dku|1+\theta|}{(8\gk)^\frac12}
		\right) 
		\tnorm{u-\uN}^2
		\\
		&\le 
		\max_{\kappa \in \cK}
		\left(
			1+\tfrac{\dku|1+\theta|}{(8\gk)^\frac12}
		\right) 
		\left(
		\left(
			\sum_{\kappa\in\cK}
			\Big[ \etaR+\etaF+\etaJ \Big]^2
		\right)^\frac12
		+ \hotub
		\right)^2,
	\end{align*}
	which leads to the final result.
\end{proof}

\begin{thm}
		Let $u\in \Hper\cap H^2(\cK)$ and $\lambda$ be the solution of \eqref{eq:EigenVal} and $u_N\in\VN$ and $\lN$ the DG-approximation defined by \eqref{eq:DG_EigenVal}.
		 Then,
		if the stabilization parameter $\gk$ is large enough, i.e. $\gk \ge \frac12 (1+\theta)^2 (\dku)^2$, and the high order terms not dominating, i.e. $ 2\,\lambda\, \|u-\uN\|_\Omega^2 < \tnorm{u-\uN}^2$, then,
		 we have the following a posteriori lower bound for the
     approximation error in the eigenvalue 
	\[
		\frac12 \xi^2
		\le 
		|\lN - \lambda|
		+ 
		\lambda\, \|u-\uN\|_\Omega^2
		+ \frac12 (\hotlb)^2.
	\]	
\end{thm}
\begin{proof}
We first observe that 
\[
	a(u-\uN,u-\uN) \ge \frac12 \tnorm{u-\uN}^2,
\]
under the first assumption, i.e. that $\gk \ge \frac12 (1+\theta)^2 (\dku)^2$. Indeed, the proof is identical to the one of Lemma 3.1 of \cite{LinStamm2015} by replacing the arbitrary discrete function $v_N$ by the error function $u-\uN$ and using the constant $\dku$ instead of $\dkN$. 

Then, starting from \eqref{eq:eigenval_eq} we see that 
\[
	|\lN - \lambda| 
	=
  \big| a(u-\uN,u-\uN) -  \lambda\, \|u-\uN\|_\Omega^2\big |
\]
Now, observing that the second assumption of the Theorem implies that 
\[
	a(u-\uN,u-\uN) \ge \frac12 \tnorm{u-\uN}^2 \ge \lambda\, \|u-\uN\|_\Omega^2,
\]
we deduce that 
\[
	|\lN - \lambda| 
	\ge 
	 a(u-\uN,u-\uN) - \lambda\, \|u-\uN\|_\Omega^2
	\ge 
	\frac12 \tnorm{u-\uN}^2 - \lambda\, \|u-\uN\|_\Omega^2.
\]
Finally, we deduce the final result by applying Proposition \ref{prop:GlobLowerBound} to obtain a lower bound of the energy error.
\end{proof}

%% file: NumericalResults.tex
\section{Numerical results}\label{sec:numer}

In this section we test the effectiveness of the a posteriori error
estimators.  The test program is written in MATLAB, and all results are
obtained on a 2.7 GHz Intel processor with 16 GB memory.  
All numerical results are performed using the symmetric bilinear form
($\theta=1$).

The error in the energy norm of the $i$-th eigenfunction is denoted by
$\tnorm{u_{i}-u_{i,N}}$. 
We will compare $\tnorm{u_{i}-u_{i,N}}$ with
our parameter-free upper bound estimator $\eta_{i}$ and lower bound
estimator $\xi_{i}$, respectively. For the eigenvalues, our theory in
Section~\ref{sec:EigenVal} indicates that after neglecting the high
order terms, the upper bound for the error of the $i$-th eigenvalue
$\abs{\lambda_{i}-\lambda_{i,N}}$ can be taken as $C_{1}\eta_{i}^2$, and the
lower bound should be $\xi_{i}^2/C_{2}$, where $C_{1},C_{2}$ are
positive constants larger than $1$. However, our estimate of the error
of the eigenvalues is based on the estimate of the error of the
eigenfunctions, and hence the upper and lower bound estimators for
eigenvalues may deviate further from the true error of eigenvalues.
Numerical results below indicate that it is possible to choose and use
$\eta_{i}^2$ and $\xi_{i}^2$ as the \textit{numerical} upper and lower
bound estimator, for the error of the $i$-th eigenvalue,
respectively, i.e. setting $C_{1}=C_{2}=1$. 

The definition of the energy norm contains the term
$\|V^\frac12 (u_{i}-u_{i,N})\|^2_\kappa$. This
term characterizes a weighted $L^{2}$ error of the eigenfunction, and
hence is asymptotically less important than the rest of the terms in the
energy error. Nonetheless we include this term explicitly in the
computation, where $V^{\frac12}$ is replaced by $(V-V_{m})^{\frac12}$,
and $V_{m}$ the minimum of the potential $V$ in $\Omega$. As mentioned
in Section~\ref{sec:EigenVal}, such shift is possible since the addition
of a constant only shifts all eigenvalues by a constant, without changing the
eigenfunctions.
In the numerical computation, the intuitively high order terms
$\hotub$ and $\hotlb$ that are part of
the upper and lower bound estimators, which were derived in
Section~\ref{sec:EigenVal}, are neglected. 
Although we do not have a priori error analysis
for general non-polynomial basis functions to justify that such terms
are indeed of higher order compared to the upper and lower bound
estimators, respectively, we compute these terms explicitly. As we will see
in the numerical examples, $\hotub,\hotlb$ can indeed be much smaller
than the upper and lower bound estimators, respectively, when the
approximate solution converges to the true solution as the basis set
is enriched. 


Our test systems are selected from the same set as those used in Part I of
this manuscript~\cite{LinStamm2015}. Numerical results indicate that our
estimators for eigenfunctions capture the true error within a factor
$2\sim 5$, across a wide range of accuracy. Since the error of
eigenvalues is on the order of magnitude of the square of the error of
eigenfunctions, our upper and lower bound estimators for eigenvalues is
generally within an order of magnitude of  the error of the eigenvalues. 

As discussed in Section~\ref{sec:EigenVal}, it is straightforward to
measure the error of eigenvalues. 
Special care should be taken when measuring the error of eigenfunctions.
Even when all eigenvalues are
simple (i.e. non-denegerate), the computed eigenfunctions may carry an arbitrary phase factor
$\pm 1$. If the multiplicity of an eigenvalue is larger than $1$, the
resulting eigenfunctions may be an arbitrary normalized vector in the
corresponding eigenspace.  Therefore when measuring the error of
eigenfunctions, a ``subspace alignment'' procedure is first performed.
Assume we would like to compute the first $m$ eigenfunctions. In each
element $\kappa$, we represent the solution on a fine set of
Legendre-Gauss-Lobatto (LGL) grid points. With some abuse of notation,
we denote by $u_{i}$, for $i=1,\ldots,m$, a column vector, and each entry of
the vector is the value of the true eigenfunction evaluated on one
such LGL grid point. This setup is the same as that used
in~\cite{LinStamm2015}. We also denote by $W$ a diagonal matrix with
each diagonal entry being the quadrature weight associated with a LGL
grid point, such that the discrete normalization condition can be
written as
\[
u_{i}^{T}W u_{j} = \delta_{ij}, \quad 1\le i,j\le m.
\]
Here $\delta_{ij}$ is the Kronecker $\delta$-symbol. Similarly
$u_{i,N}$ denotes the column vector with each entry being the value of
the approximate eigenfunction in the DG method evaluated on a LGL grid point, and satisfies
the normalization condition
\[
u_{i,N}^{T}W u_{j,N} = \delta_{ij}, \quad 1\le i,j\le m.
\]

Define the matrix $U=[u_{1},\ldots,u_{m}]$ 
and $U_{N}=[u_{1,N},\ldots,u_{m,N}]$. Then we define the aligned eigenfunctions,
denoted by $\wt{U}_{N}=[\wt{u}_{1,N},\ldots,\wt{u}_{m,N}]$, as
\begin{equation}
  \wt{U}_{N} = U_{N} (U_{N}^{T}WU).
  \label{eq:alignment}
\end{equation}
When $m=1$, Eq.~\eqref{eq:alignment} reduces to
\[
\wt{u}_{1,N} = u_{1,N} (u_{1,N}^{T}Wu_{1}), 
\]
and the subspace alignment procedure can clearly recover the potential
phase factor discrepancy when $u_{1}$ and $u_{1,N}$.
Eq.~\eqref{eq:alignment} can be further used when certain eigenvalues
are degenerate.
Then in practice,
$\tnorm{u_{i}-u_{i,N}}$ is computed from $\tnorm{u_{i}-\wt{u}_{i,N}}$. With slight abuse of notation, in the discussion below
$u_{i,N}$ refers to the aligned eigenfunction $\wt{u}_{i,N}$. All
eigenfunctions have normalized $2$-norm in the real space, and therefore
the order of magnitude of absolute errors of eigenfunctions is also
comparable to that of the relative errors.

The quality of the upper and lower bound estimators for the $i$-th eigenfunction
is measured by
\[
C_{i,\eta} = \frac{\eta_{i}}{\tnorm{u_{i}-u_{i,N}}}, \quad 
C_{i,\xi}  = \frac{\xi_{i}}{\tnorm{u_{i}-u_{i,N}}},
\]
respectively. The estimators are strictly upper and lower bound of the
error if $C_{i,\eta}>1$ and $C_{i,\xi}<1$, and the estimators are
considered to be effective if they are close to $1$.
Similarly, the estimators for
the eigenfunctions are defined to be
\[
C^{\lambda}_{i,\eta} =
\frac{\eta_{i}^2}{|\lambda_{i}-\lambda_{i,N}|}, \quad 
C^{\lambda}_{i,\xi}  = \frac{\xi_{i}^2}{|\lambda_{i}-\lambda_{i,N}|}.
\]

Our test problems include both one dimensional (1D) and two dimensional
(2D) domains with periodic boundary conditions. The numerical
examples are chosen to be the two difficult cases in our previous
publication~\cite{LinStamm2015}.
Our non-polynomial
basis functions are generated from the adaptive local basis (ALB)
set~\cite{LinLuYingE2012} in the DG framework. The ALB set was proposed
to systematically reduce the number of basis functions used to solve
Kohn-Sham density functional theory calculations, which involves large
scale eigenvalue computations. 

We denote by $N$ the number of ALBs per
element.
For operators in the form of $A=-\Delta+V$ with periodic boundary
condition, the basic idea of the ALB set is to use eigenfunctions
computed from local domains as basis functions corresponding to the lowest
few eigenvalues.  The eigenfunctions are associated with the same
operator $A$, but with modified boundary conditions on the local domain.
More specifically, in
a $d$-dimensional space, for each element $\kappa$, we form an
\textit{extended element} $\widetilde{\kappa}$ consisting of $\kappa$
and its $3^{d}-1$ neighboring elements in the sense of periodic boundary
condition. On $\widetilde{\kappa}$ we solve the eigenvalue problem
\begin{equation}
  -\Delta \widetilde{\varphi}_{i} + V \widetilde{\varphi}_{i} =
  \lambda_{i},
  \widetilde{\varphi}_{i}.
  \label{}
\end{equation}
with periodic boundary condition on $\partial \widetilde{\kappa}$. The
collection of eigenfunctions (corresponding to lowest $N$ eigenvalues)
are restricted from $\widetilde{\kappa}$ to $\kappa$, i.e. 
\[
\varphi_{i}(x) = \begin{cases}
  \widetilde{\varphi}_{i}
(x), & x\in \kappa;\\
  0,&\text{otherwise}.
\end{cases}
\]
After orthonormalizing the set of basis functions $\{\varphi_{i}\}_{i=1}^N$ locally on each element
$\kappa$ and removing the linearly dependent functions, the resulting
set of orthonormal functions are called the ALB functions.  

Since periodic boundary conditions are used on the global domain $\Omega$,
the reference solution is solved using a planewave basis set
with a sufficiently large number of planewaves.  The ALB set is also
computed using a sufficiently large number of planewaves on the extended
element $\widetilde{\kappa}$. Then a Fourier interpolation procedure
is carried out from $\widetilde{\kappa}$ to the local element
LGL for accurate numerical
integration. 

\subsection{1D example}\label{subsec:ex1D}

We first demonstrate the effectiveness of the a posteriori error
estimates for a second order operator on a 1D domain
$\Omega=[0,2\pi]$, using the ALB set as non-polynomial basis functions.
The potential function $V(x)$ is given by the sum
of three Gaussian functions with negative magnitude, as shown in
Fig.~\ref{fig:uuerr1D} (a).  The operator $A=-\Delta+V$ has $3$
negative eigenvalues and is indefinite.  The domain is partitioned into
$7$ elements for the ALB calculation.  Fig.~\ref{fig:uuerr1D} (b)
shows the first eigenfunction $u_{1}$, and
Fig.~\ref{fig:uuerr1D} (c) shows the point-wise error $u_{1}-u_{1,N}$
using $N=6$ ALBs per element.

\begin{figure}[h]
  \begin{center}
    \subfloat[]{\includegraphics[width=0.25\textwidth]{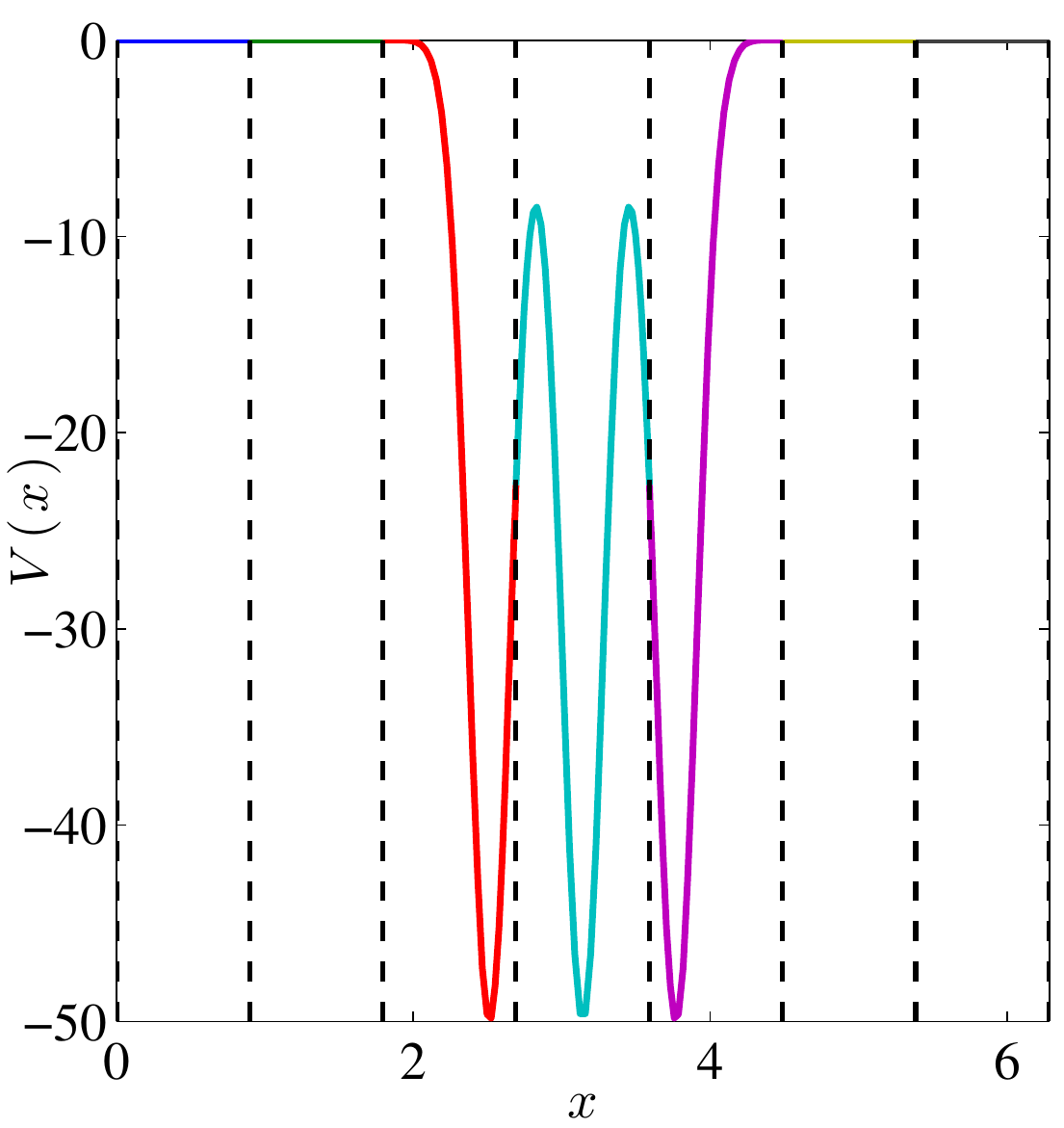}}
    \quad
    \subfloat[]{\includegraphics[width=0.25\textwidth]{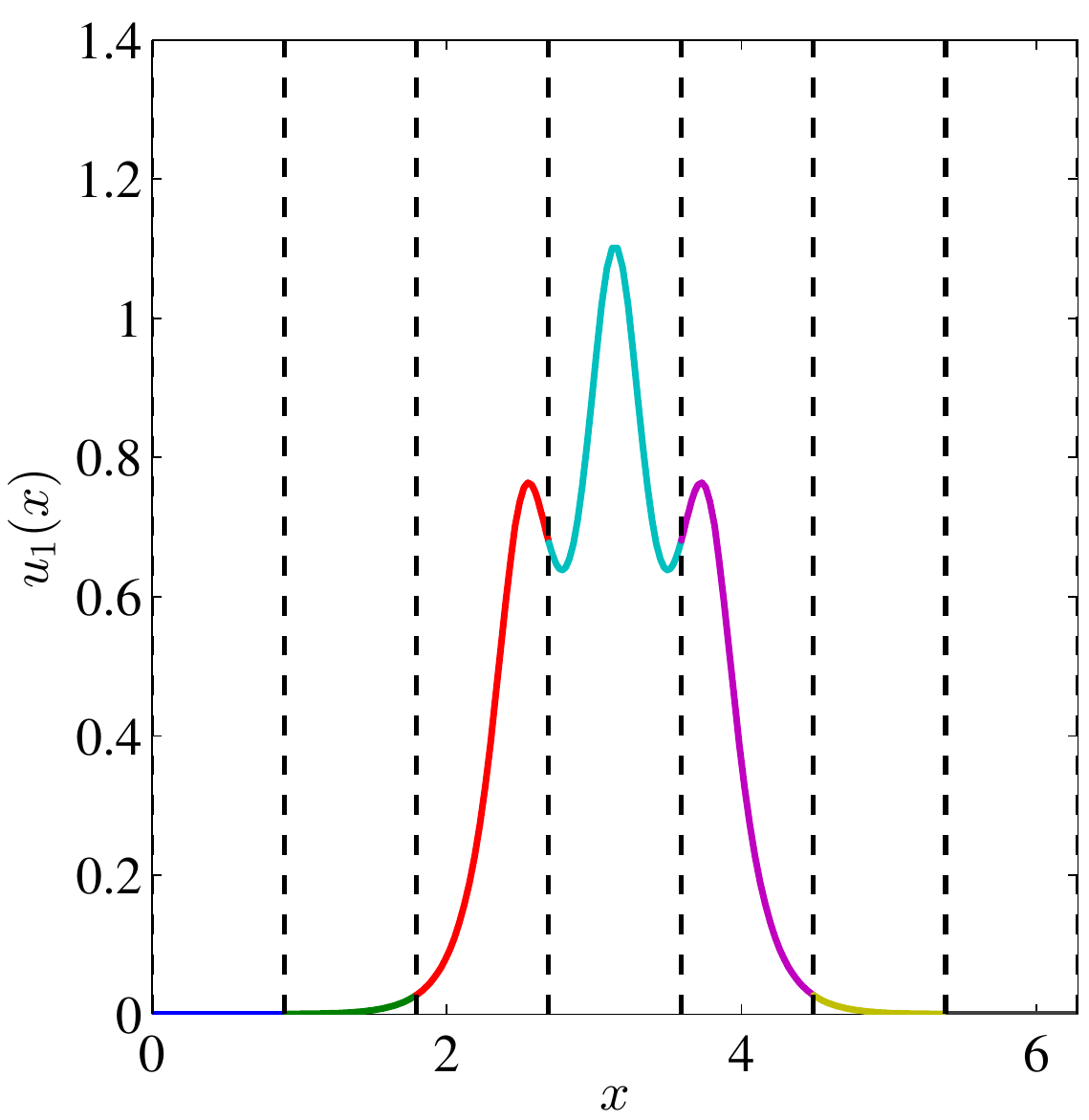}}
    \quad
    \subfloat[]{\includegraphics[width=0.24\textwidth]{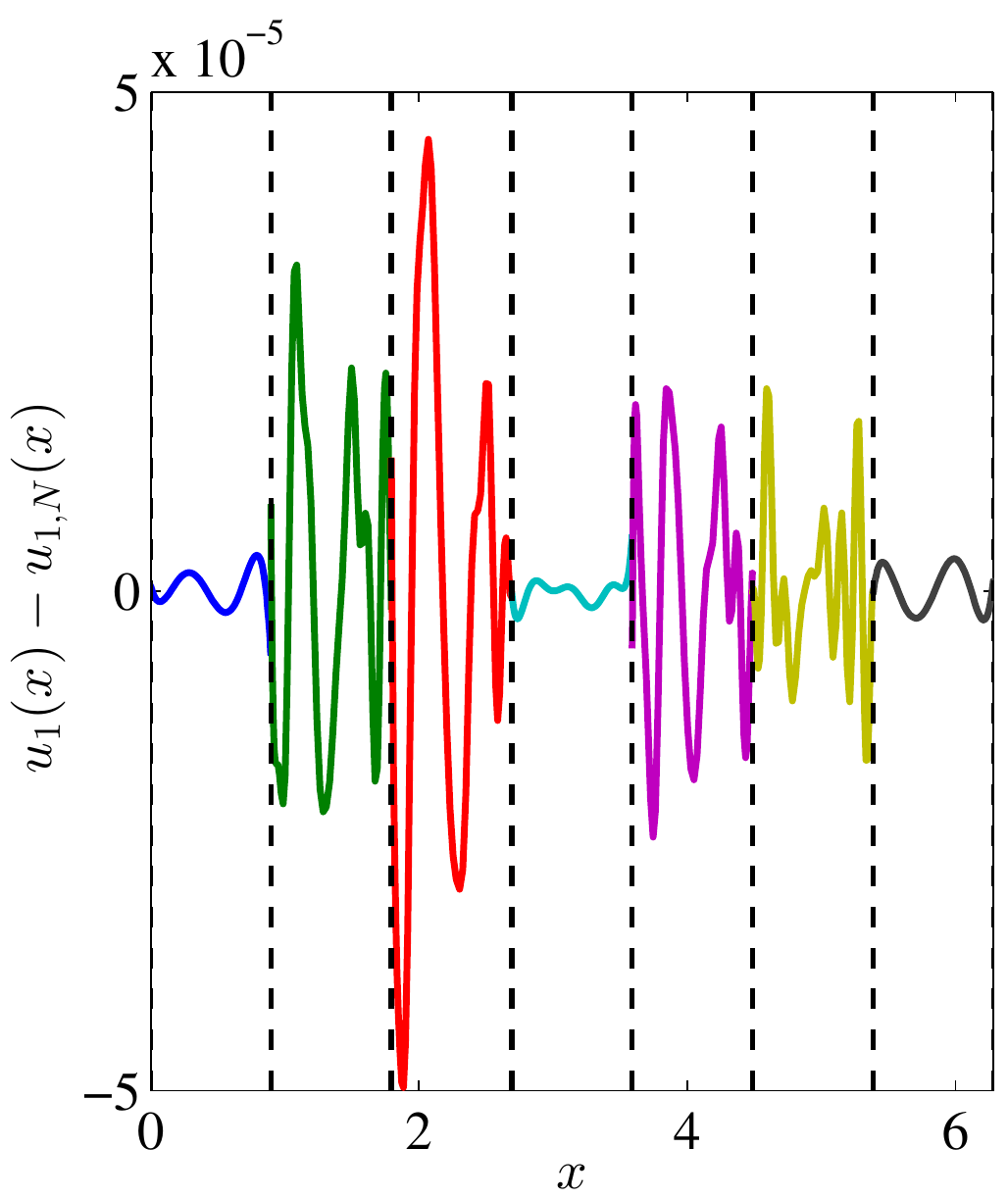}}
  \end{center}
  \caption{(a) The potential $V(x)$ given by the sum of three Gaussians
  with negative magnitude.  (b) The first eigenfunction $u_{1}(x)$. (c)
  Point-wise error between the first eigenfunction $u_1(x)$ and the numerical
  solution $u_{1,N}(x)$ calculated using the ALB set with $7$ elements and
  $N=6$ basis functions per element.}
  \label{fig:uuerr1D}
\end{figure}


Fig.~\ref{fig:1Destimator} (a), (b) compare the error of the first $11$
eigenvalues and the corresponding eigenfunctions, together with the
upper and lower estimators, respectively, using a relatively small
number of $6$ basis functions per
element. For the eigenfunctions, $C_{i,\eta}$ ranges from $1.50$ to
$2.82$. Hence $\eta_{i}$ is indeed an effective upper bound for
$\tnorm{u_{i}-u_{i,N}}$.  The lower bound estimator $C_{i,\xi}$ ranges
from $0.32$ to $0.58$, and therefore is effective as well. In terms of
eigenvalues, the upper bound estimator $C^{\lambda}_{i,\eta}$ ranges
from $3.16$ and $9.70$, and the lower bound estimator for eigenvalues
$C^{\lambda}_{i,\xi}$ ranges from $0.17$ to $0.40$. While the upper and
lower bound of the eigenvalues remains to be true upper and lower bound,
respectively, we note that the eigenvalue estimator is less effective
compared to that of the eigenfunctions, and the upper (lower) bound
estimator can overestimate (underestimate) the error by around one order
of magnitude. Nonetheless, we note in Fig.~\ref{fig:1Destimator} (a)
that the error of eigenvalues spans over $4$ orders of magnitude, and
our upper and lower estimators well captures such inhomogeneity in terms
of accuracy among
the different eigenvalues. The same trend is observed for eigenfunctions
in Fig.~\ref{fig:1Destimator} (b). Fig.~\ref{fig:1Destimator} (b) also
reports the terms $\hotub,\hotlb$ defined in Section~\ref{sec:EigenVal}.
We find that $\hotub_{i}$ and $\hotlb_{i}$ are significantly smaller than
$\eta_{i}$ and $\xi_{i}$, respectively, and thus justify numerically that
such terms are indeed high order terms.


Fig.~\ref{fig:1Destimator} (c), (d) demonstrate the error of
eigenvalues and eigenfunctions and the associated estimators using a
more refined basis set, with $10$
basis functions per element. Despite the small increase of the number of
basis functions, the error of eigenvalues is reduced to as low as
$10^{-8}$. $C_{i,\eta}$ for eigenfunctions is between $2.19$ and $2.45$,
and $C_{i,\xi}$ is between $0.64$ and $0.67$. The effectiveness
parameters are remarkably homogeneous for all eigenfunctions computed.
Correspondingly $C^{\lambda}_{i,\eta}$ for eigenvalues is between $5.27$
and $7.37$, and $C^{\lambda}_{i,\xi}$ for eigenvalues is between
$0.45$ and $0.58$. The difference between $\eta_{i},\xi_{i}$ compared to
$\hotub_{i},\hotlb_{i}$ is amplified even further in
Fig.~\ref{fig:1Destimator} (d) as the basis set refines, and therefore
justifies that $\hotub_{i},\hotlb_{i}$ are indeed of higher order.
%

\begin{figure}[h]
  \begin{center}
    \subfloat[$6$
    basis]{\includegraphics[width=0.25\textwidth]{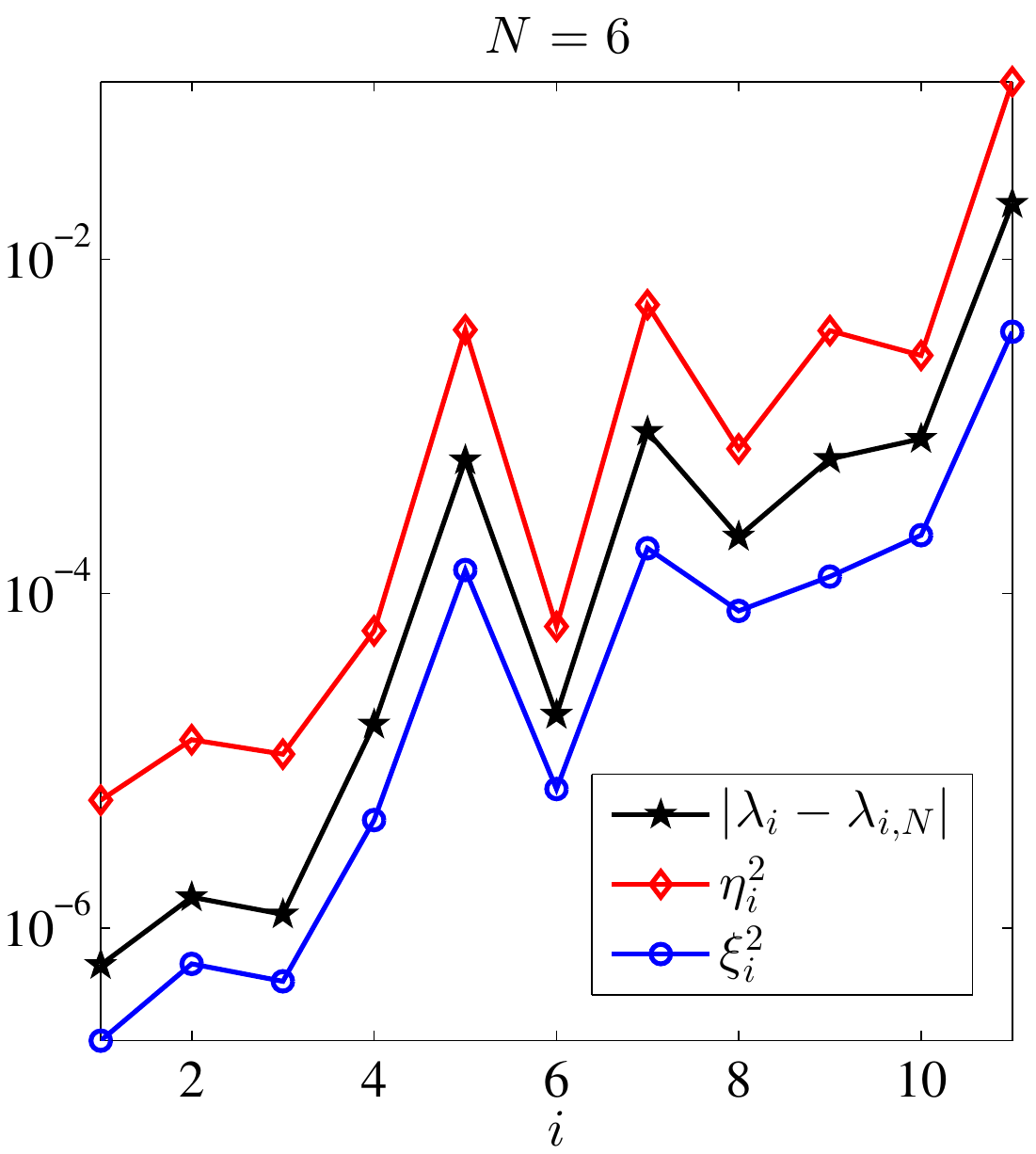}}
    \quad
    \subfloat[$6$
    basis]{\includegraphics[width=0.25\textwidth]{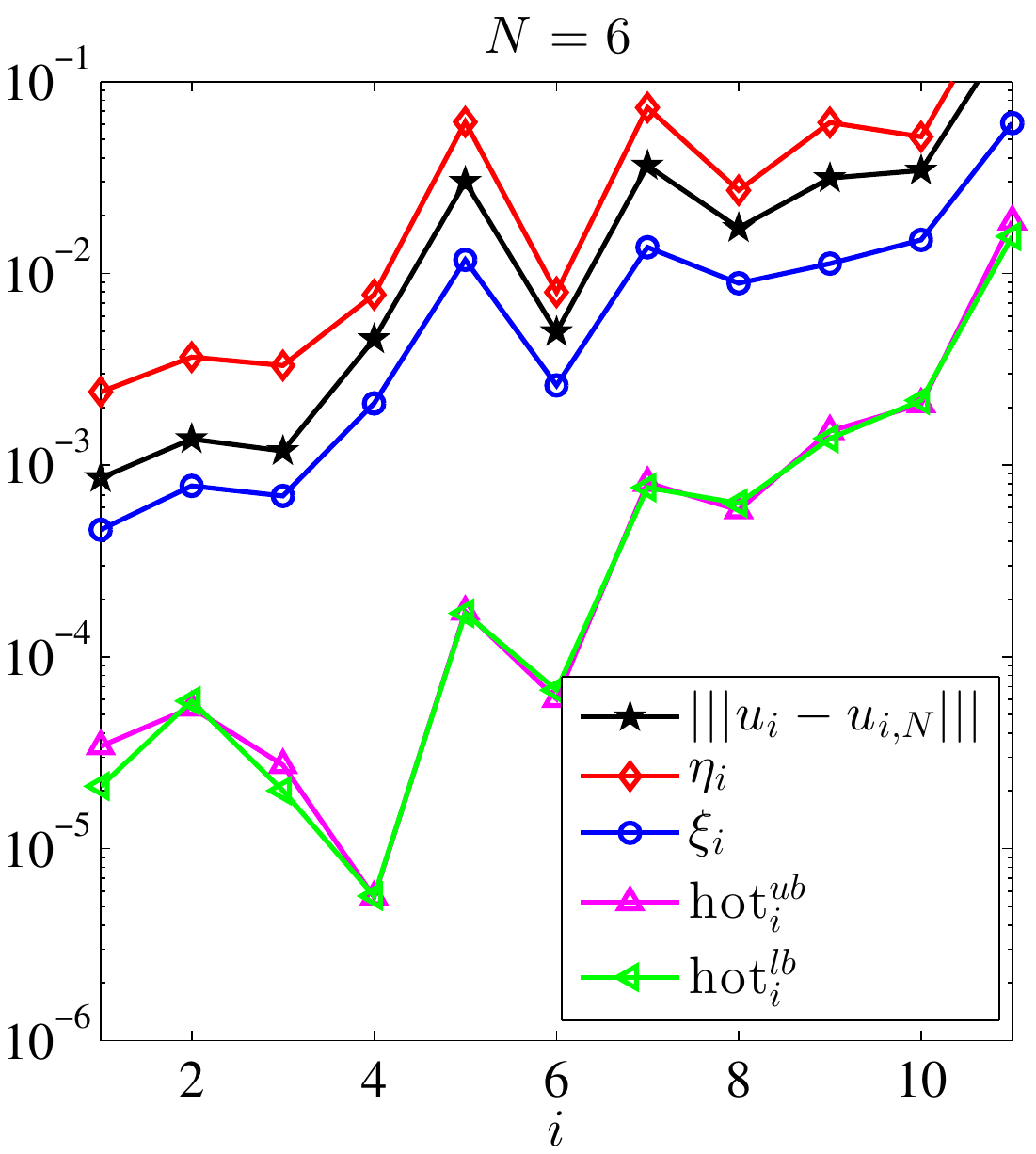}}\\
    \subfloat[$10$
    basis]{\includegraphics[width=0.25\textwidth]{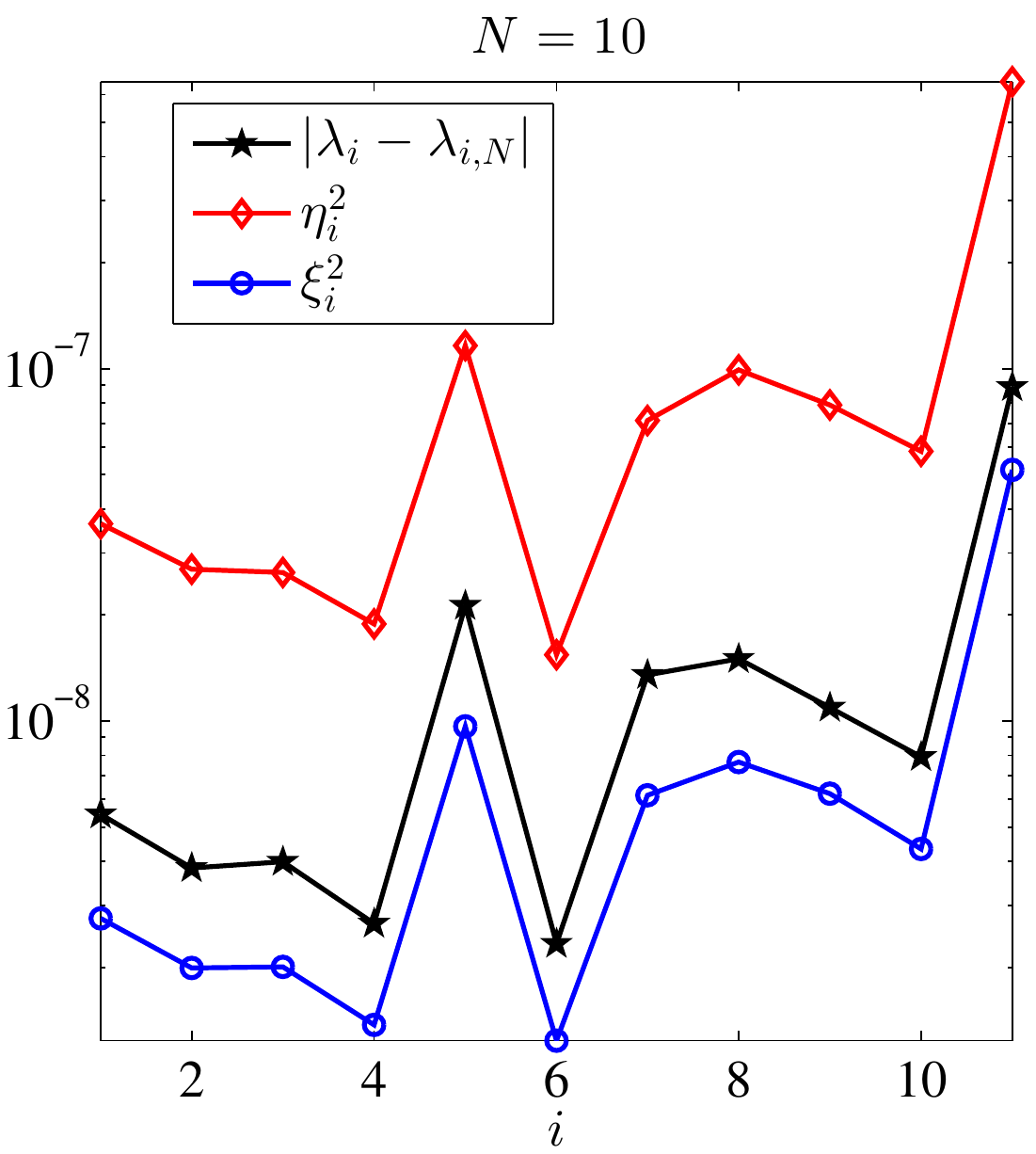}}
    \quad
    \subfloat[$10$
    basis]{\includegraphics[width=0.25\textwidth]{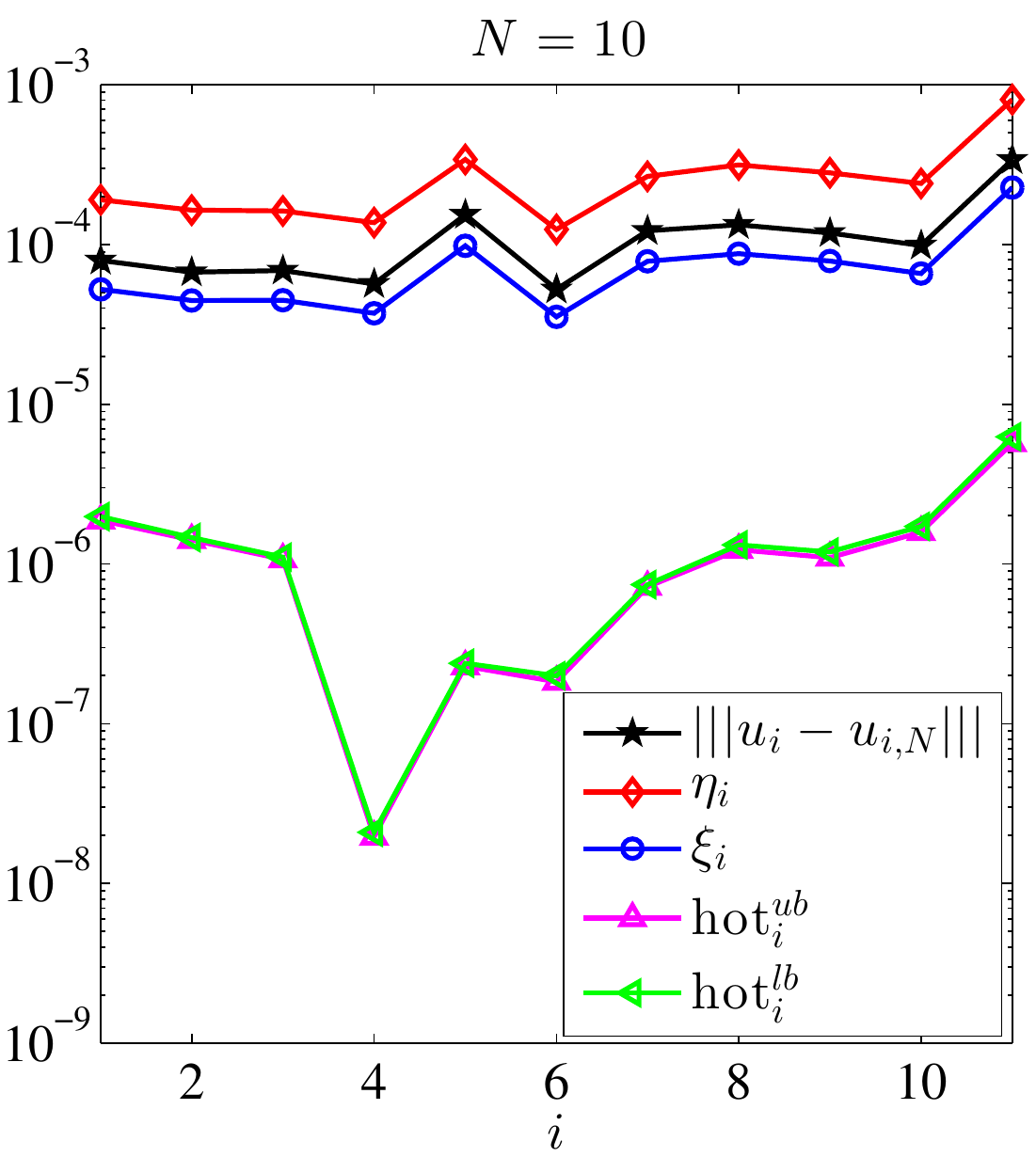}}
  \end{center}
  \caption{Error of the (a) eigenvalues and (b) eigenfunctions together
  with upper and lower bound estimator for the first $11$
  eigenfunctions, using $6$ basis functions per element. (c),(d) are the
  same as (a),(b) respectively but with $10$ basis functions per
  element.  }
  \label{fig:1Destimator}
\end{figure}

\subsection{2D example}\label{subsec:ex2D}

Our second example is a 2D problem on $\Omega=[0,2\pi]\times [0,2\pi]$
with periodic boundary condition. The potential $V$ is given by the sum
of four Gaussians with negative magnitude, as illustrated in
Fig.~\ref{fig:uuerr2D} (a). Fig.~\ref{fig:uuerr2D} (b) shows
the first eigenfunction $u_{1}$ and Fig.~\ref{fig:uuerr2D} (c) shows the
point-wise error $u_{1}-u_{1,N}$ using $N=11$ ALBs per element.   In the
ALB computation, the domain is partitioned into $5\times 5$ elements,
indicated by black dashed lines. 

\begin{figure}[h]
  \begin{center}
    \subfloat[]{\includegraphics[width=0.25\textwidth]{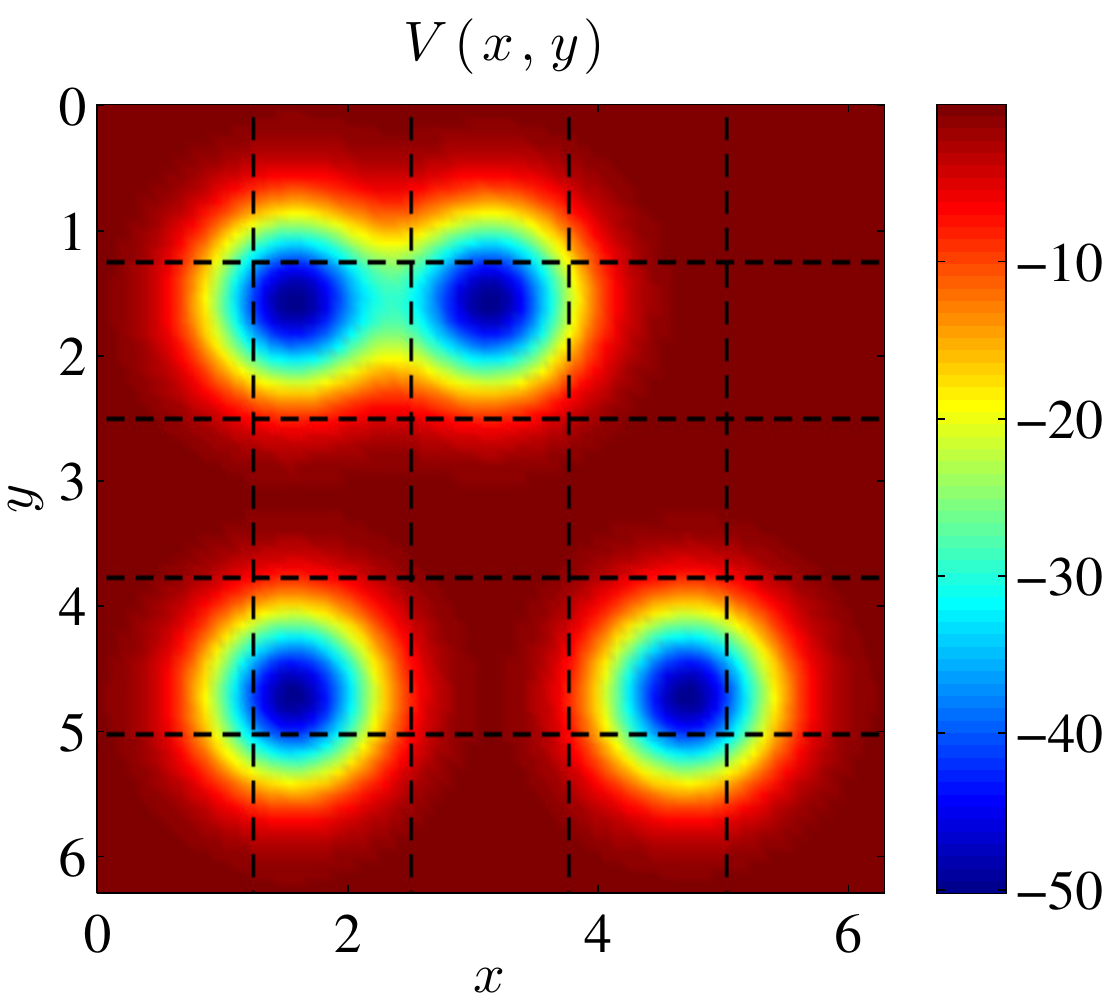}}
    \quad
    \subfloat[]{\includegraphics[width=0.25\textwidth]{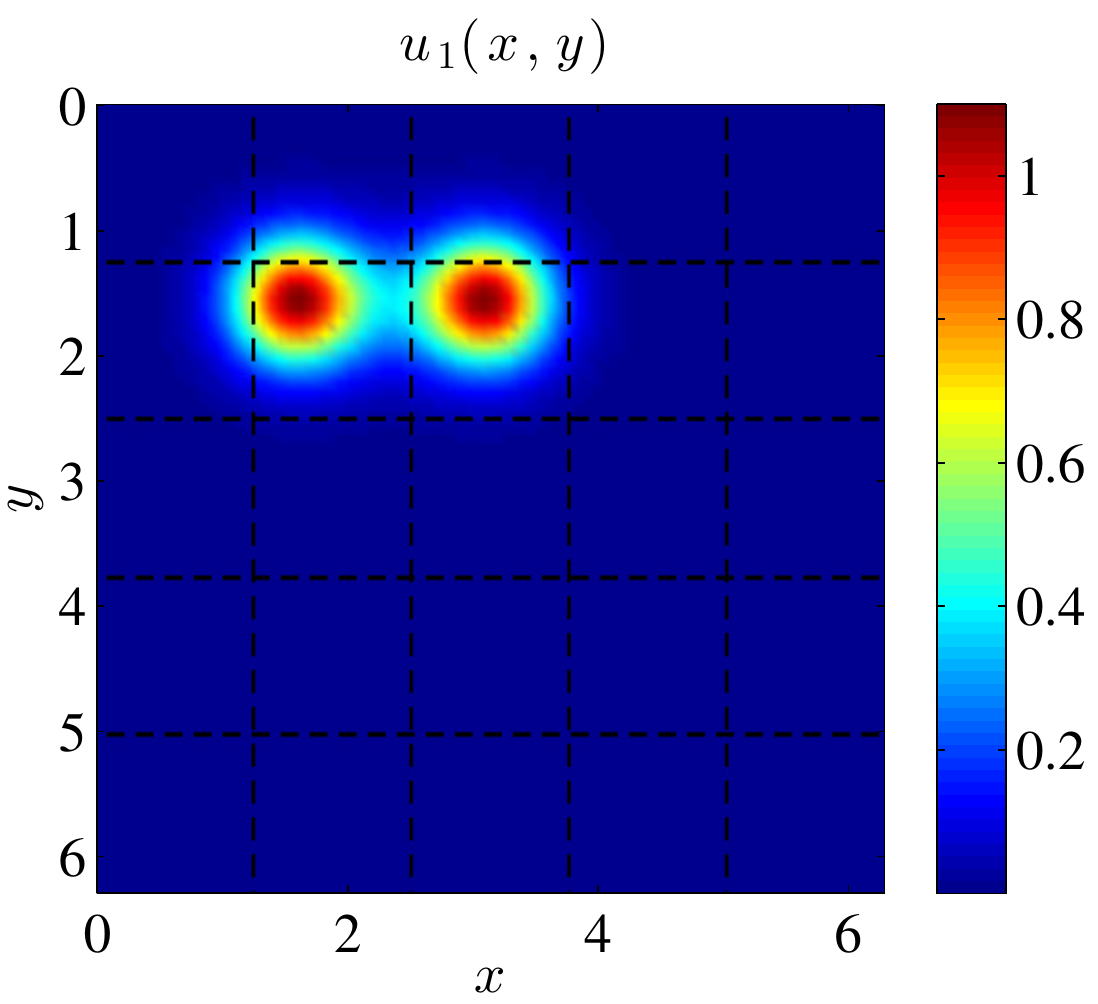}}
    \quad
    \subfloat[]{\includegraphics[width=0.24\textwidth]{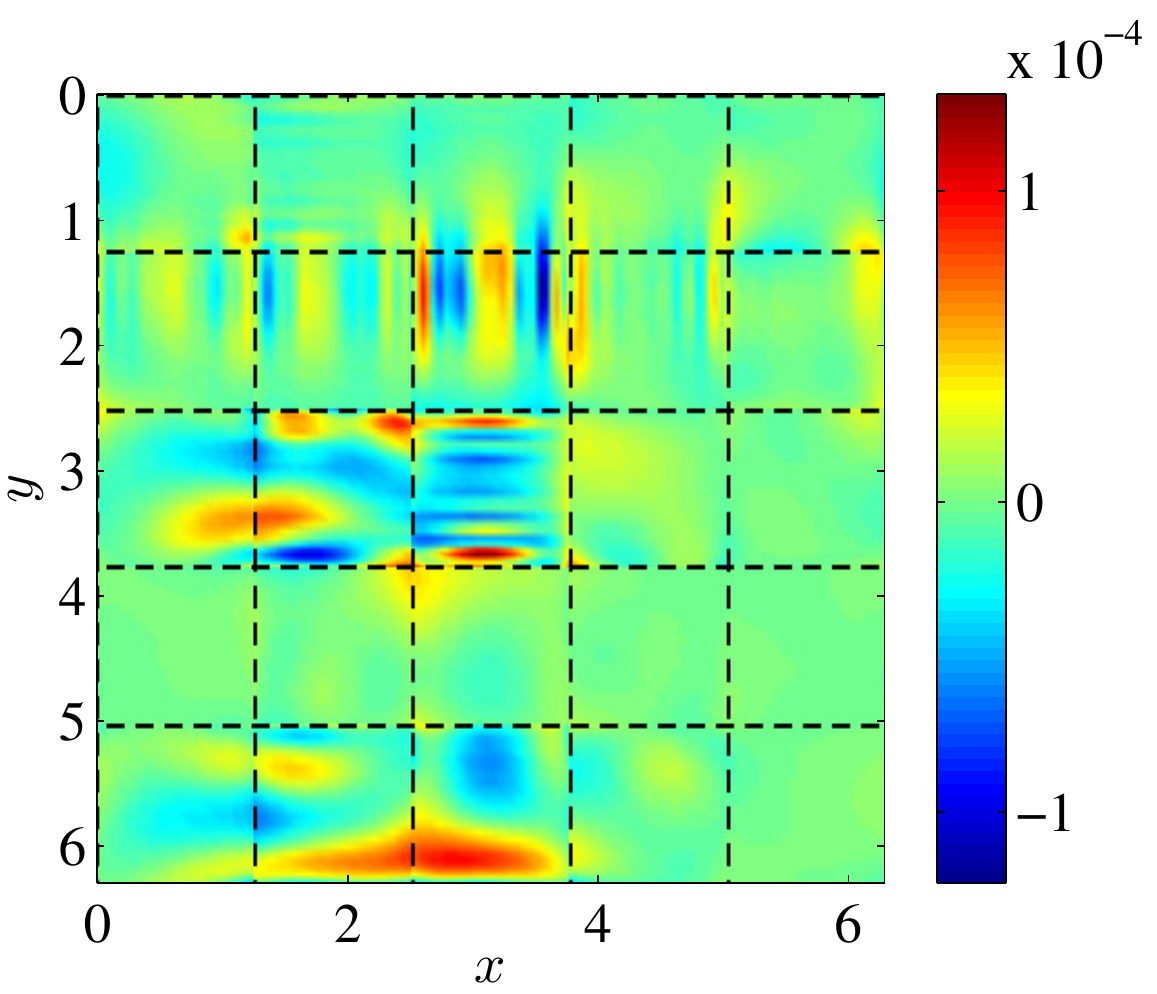}}
  \end{center}
  \caption{(a) The potential $V(x,y)$.  (b) The first eigenfunction $u_{1}(x)$. (c)
  Point-wise error between the first eigenfunction $u_1(x)$ and the numerical
  solution $u_{1,N}(x)$ calculated using the ALB set with $5\times 5$ elements and
  $N=11$ basis functions per element.}
  \label{fig:uuerr2D}
\end{figure}


Similar to the 1D case, Fig.~\ref{fig:2Destimator} (a), (b) 
compare the error of the first $11$
eigenvalues and the corresponding eigenfunctions, together with the
upper and lower estimators, respectively, using $11$ basis functions per
element. The effectiveness parameter for eigenfunctions $C_{i,\eta}$
ranges from $2.78$ to
$4.59$, and the $C_{i,\xi}$ ranges from $0.22$ to
$0.36$. For the eigenvalues, the upper
bound estimator $C^{\lambda}_{i,\eta}$ is between $12.73$ and $21.68$,
and the lower bound estimator for eigenvalues $C^{\lambda}_{i,\xi}$
is between $0.08$ to $0.14$. Similarly, we observe that
$C^{\lambda}_{i,\eta}$ and $C^{\lambda}_{i,\xi}$ are roughly on the
order of magnitude of the square of
the $C_{i,\eta}$ and $C_{i,\xi}$, respectively.  Again our upper and
lower bound estimator well captures the large inhomogeneity in terms of
accuracy among different eigenvalues and eigenfunctions.

Fig.~\ref{fig:2Destimator} (c), (d) show the error of
eigenvalues and eigenfunctions and the associated estimators using a
large number of $41$ basis functions per element.  $C_{i,\eta}$ for eigenfunctions is
between $2.00$ and $2.41$, and $C_{i,\xi}$ is between $0.23$ and $0.32$.
The effectiveness parameters are remarkably homogeneous for all
eigenfunctions computed.  Correspondingly $C^{\lambda}_{i,\eta}$ for
eigenvalues is between $4.45$ and $6.85$, and $C^{\lambda}_{i,\xi}$ for
eigenvalues is between $0.06$ and $0.11$. The high order terms
$\hotub_{i},\hotlb_{i}$ are reported in Fig.~\ref{fig:2Destimator} (b)
and (d). Again we find that such terms are smaller than the
upper and lower estimators, and the difference become more enhanced as
the basis set refines.

\begin{figure}[h]
  \begin{center}
    \subfloat[$11$
    basis]{\includegraphics[width=0.25\textwidth]{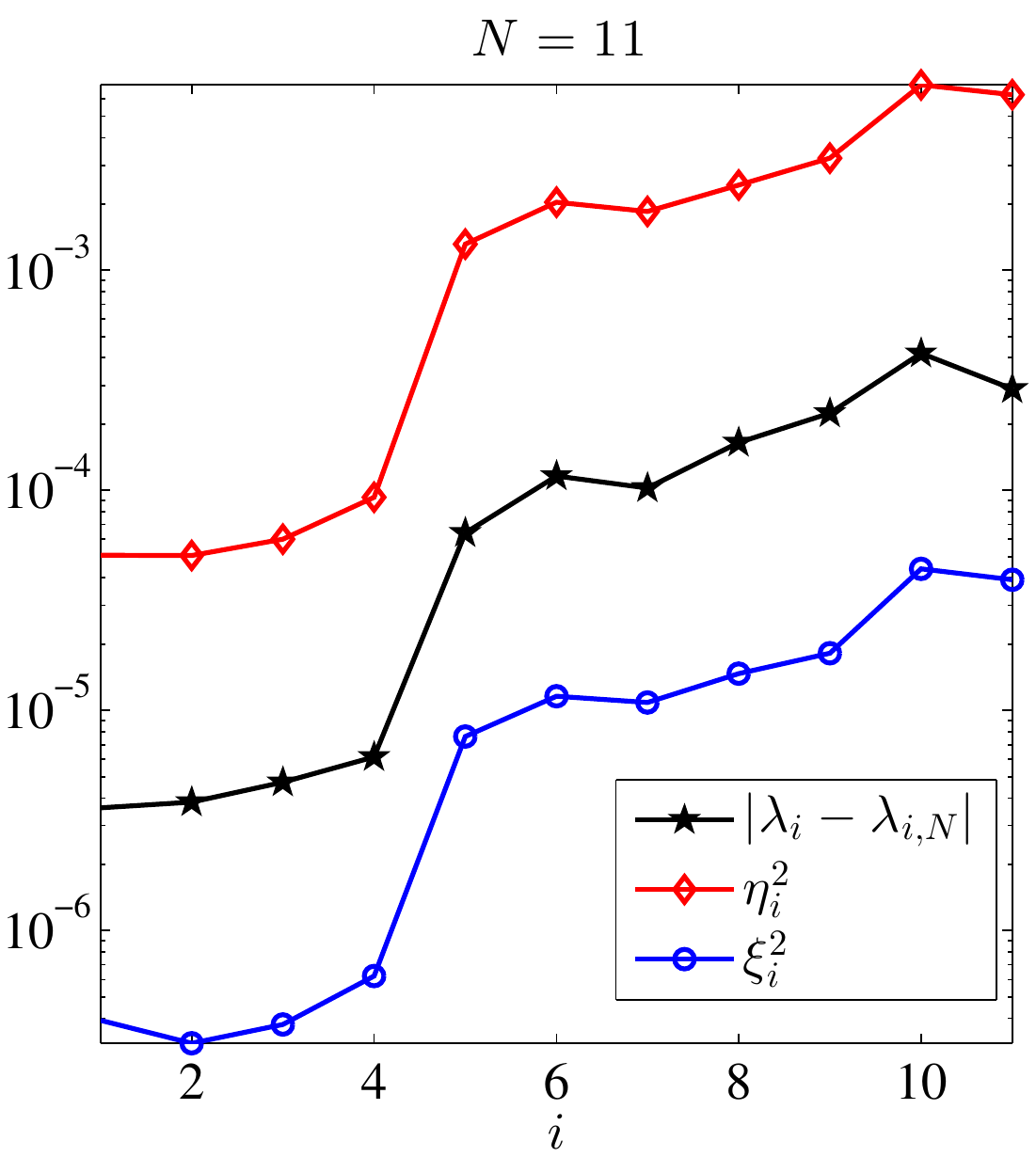}}
    \quad
    \subfloat[$11$
    basis]{\includegraphics[width=0.25\textwidth]{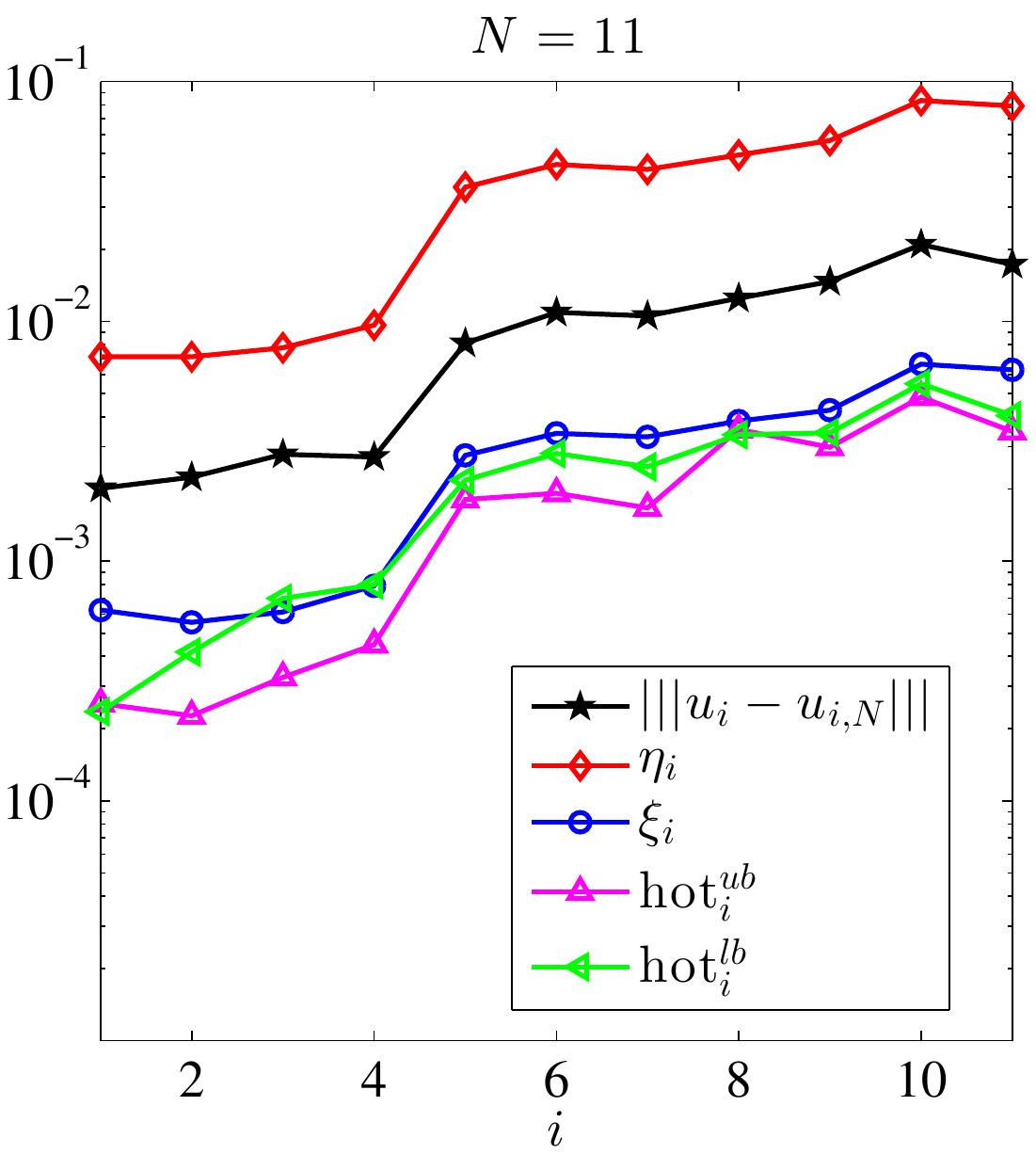}}\\
    \subfloat[$41$
    basis]{\includegraphics[width=0.25\textwidth]{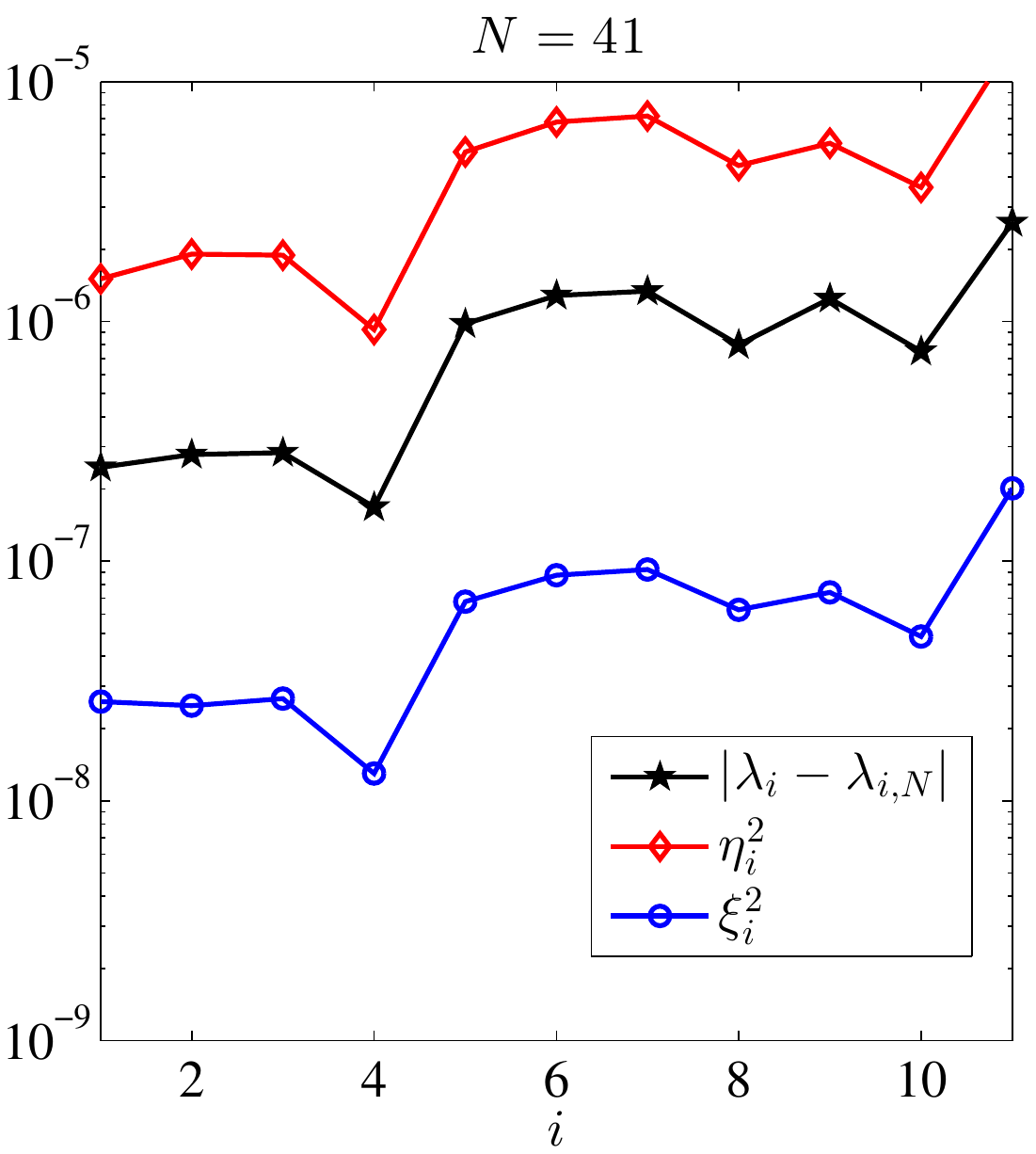}}
    \quad
    \subfloat[$41$
    basis]{\includegraphics[width=0.25\textwidth]{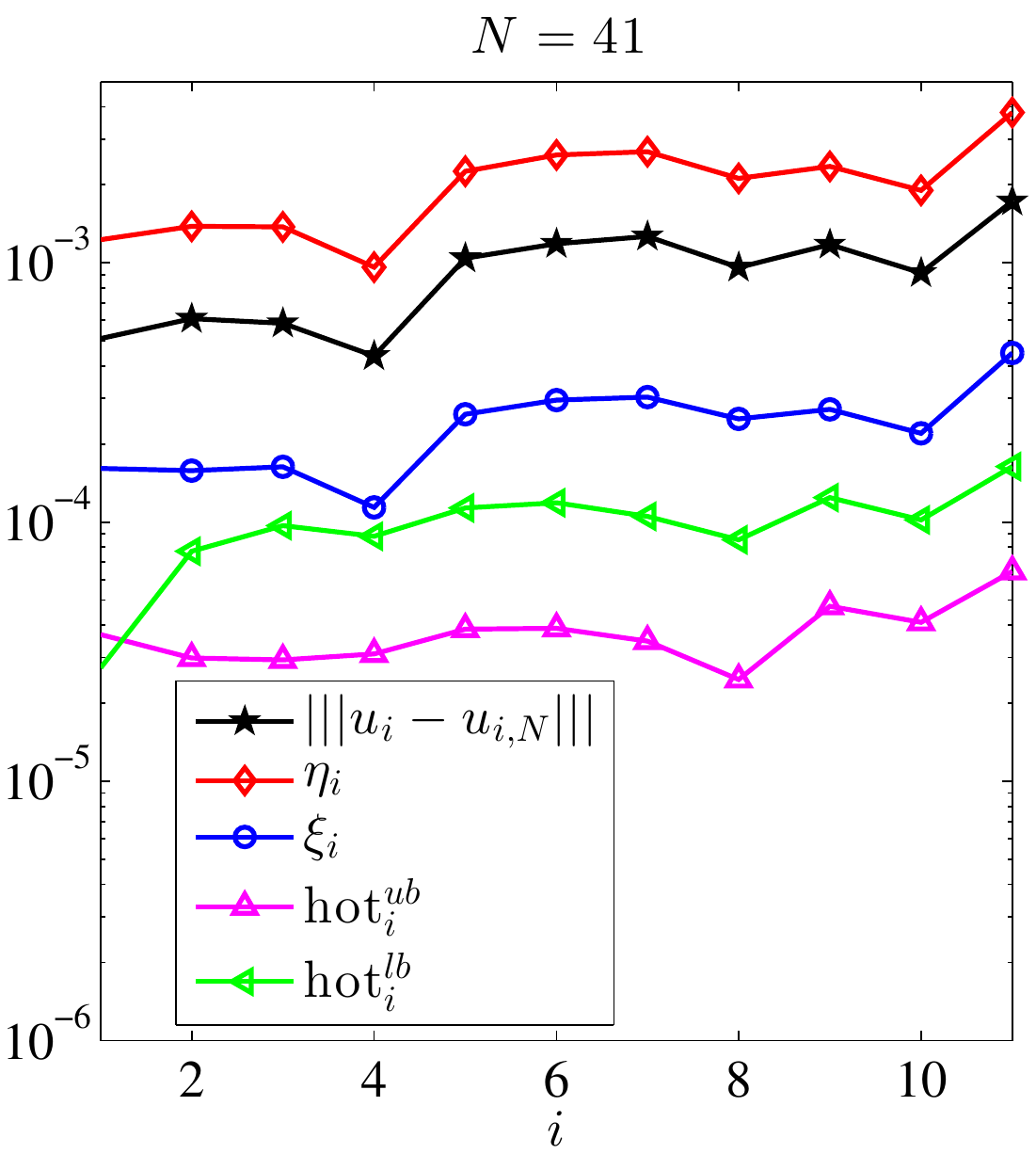}}
  \end{center}
  \caption{Error of the (a) eigenvalues and (b) eigenfunctions together
  with upper and lower bound estimator for the first $11$
  eigenfunctions, using $11$ basis functions per element. (c),(d) are the
  same as (a),(b) respectively but with $41$ basis functions per
  element.  }
  \label{fig:2Destimator}
\end{figure}

\FloatBarrier

%% file: Conclusion.tex
\section{Conclusion}
In this paper, we extend the framework that was introduced in the companion paper (Part I)  
\cite{LinStamm2015} to linear eigenvalue problems for second order
partial differential operators in a discontinuous Galerkin (DG)
framework. Our method provides residual type a posteriori upper
and lower bounds estimators for estimating the error of the numerically
computed eigenvalues and eigenfunctions.
The key-feature of our approach is that in absence of a priori inverse
type inequalities for non-polynomial basis functions, local eigenvalue
problems are solved and subsequently embedded in the a posteriori
estimates.  Hence our estimate is tailored for each new set of
basis functions, and numerical results illustrate the effectiveness of
our approach.  

Future developments will naturally concern the extension to non-linear eigenvalue
problems and in particular the Kohn-Sham equations in the framework
density functional theory.

\section*{Acknowledgments}

This work was partially supported by Laboratory Directed Research and
Development (LDRD) funding from Berkeley Lab, provided by the Director,
Office of Science, of the U.S. Department of Energy under Contract No.
DE-AC02-05CH11231, by the Scientific Discovery through Advanced Computing
(SciDAC) program and the Center for Applied Mathematics for Energy
Research Applications (CAMERA) funded by U.S. Department of Energy,
Office of Science, Advanced Scientific Computing Research and Basic
Energy Sciences, and by the Alfred P. Sloan fellowship (L. L.). L. L.
would like to thank the hospitality of the Jacques-Louis Lions
Laboratory (LJLL) during his visit. We sincerely thank Yvon Maday for
thoughtful suggestions and critical reading of the paper.